\documentclass[reqno,12pt]{amsart}

\usepackage{amsmath,amssymb,epsfig,color,
float,url}

\usepackage{tikz}
\usetikzlibrary{arrows,snakes,backgrounds}

\numberwithin{equation}{section}

\include{macrosb}

\newtheorem{statement}{Statement}
\newtheorem{theorem}{Theorem}
\newtheorem{propo}{Proposition}
\newtheorem{coro}{Corollary}
\newtheorem{lemma}{Lemma}
\newtheorem{definition}{Definition}
\newtheorem{remark}{Remark}

\renewenvironment{proof} {{\em Proof: \hskip 0.2cm}}{\hfill $\Box$ \\ }

\newcommand{\vufort}{{\stackrel{*}{\ \longrightarrow \ }}}
\newcommand{\simil}{\sim_{L_G}}

\newcommand{\eps}{\varepsilon}

\newcommand{\Z}{\mathbb Z}
\newcommand{\Del}[1]{\Delta^{(#1)}}
\newcommand{\Sn}[1] {{\mathcal S}_{#1}}

\newcommand{\etoile}{\stackrel{*}{\rightarrow}}

\begin{document}

\title{The Riemann-Roch  theorem for graphs
and the rank in complete graphs}

\author[Robert Cori and Yvan Le Borgne]{Robert Cori, Labri Universit\'e Bordeaux 1, \newline  Yvan Le Borgne, Labri Universit\'e Bordeaux 1 and PIMS, Simon Fraser University, Burnaby BC.}\footnote{The first   author  acknowledges the support of ERC under the agreement 
"ERC StG 208471 - ExploreMap".}\

\address{Labri, Universit\'e Bordeaux 1, 33405 Talence Cedex, France.
\hfill\break
\tt http://www.labri.fr/perso/cori/  http://www.labri.fr/perso/borgne/}
\email{robert.cori@labri.fr, borgne@labri.fr}
\date{\today}

\begin{abstract}
  The paper by M. Baker and S. Norine in 2007 introduced a new
  parameter on configurations of graphs and gave a new result in the
  theory of graphs which has an algebraic geometry flavor.  This
  result was called Riemann-Roch formula for graphs since it defines a
  combinatorial version of divisors and their ranks in terms of
  configuration on graphs.  The so called chip firing game on graphs
  and the sandpile model in physics play a central role in this
  theory.

  In this paper we give a presentation of the theorem of Baker and
  Norine in purely combinatorial terms, which is more accessible and
  shorter than the original one.

  An algorithm for the determination of the rank of configurations is
  also given for the complete graph $K_n$.  This algorithm has linear
  arithmetic complexity.  The analysis of number of iterations in a
  less optimized version of this algorithm leads to an apparently new
  parameter which we call the prerank.  This parameter and the
  parameter $\mathrm{dinv}$ provide an alternative description to some
  well known $q,t$-Catalan numbers.  Restricted to a natural subset of
  configurations, the two natural statistics degree and rank lead to a
  distribution which is described by a generating function which, up
  to a change of variables and a rescaling, is a symmetric fraction
  involving two copies of Carlitz q-analogue of the Catalan numbers.

\end{abstract}

\maketitle

We consider the following solitary  game on an undirected (non oriented) connected graph $G=(X,E)$ without loops:
at the beginning  integer  values $f_i$ are attributed to the $n$ vertices 
$x_1, x_2, \ldots x_n$ of the graph, these  values can be positive
or negative and define a configuration $f$. At each step a toppling can be performed by the player on a vertex $x_i$, it consists
in subtracting $d_i$ (the number of edges incident to $x_i$) to the amount $f_i$ and for each neighbor  $x_j$ of $x_i$  increase $f_j$  by the number of edges between these two vertices.
In this operation the amount of vertex $x_i$ may become, or stay negative. The aim of the 
player is to find  a sequence of toppling operations  which will 
end with a configuration where all the $f_i$ are non negative.   Since the sum  of the $f_i$ is invariant by  toppling, a necessary
condition to succeed is that in the initial configuration this sum should be non negative. We will see that 
this condition is not sufficient.

This game has much to do with the chip firing game (see \cite{BjoLovShor}, \cite{biggs1})
and the sandpile model  (see \cite{bakTang2},  \cite{dharIntro},  \cite{dharMajumdar}), for which  recurrent configurations 
where  defined and proved to be canonical representatives of the classes 
of configurations equivalent by a sequence of  topplings (for a more algebraic treatment see also \cite{perkPerlWilm}).

The game  was introduced and studied in detail by Baker and Norine in \cite{bakerNorine}
who also introduced a new parameter on graph configurations : the rank.  One characteristic of the 
rank $\rho(f)$ of a configuration  $f$
is that it is non negative if and only if one can get from $f$ a configuration non-negative on every vertices by performing
a sequence of topplings.
For this parameter they obtain a simple formula expressing a symmetry
similar to  the Riemann-Roch formula for surfaces and curves (a classical reference to this formula is
the book by H. Farkas and I. Kra \cite{farkasKra}).

Our aim here is to give a simple  presentation of this Riemann-Roch like theorem for all graphs and 
study the values of this parameter in the particular case where $G$ is a complete 
graph on $n$ vertices. For these complete graphs,
 it was noticed (see Proposition 2.8. in \cite{coriRossin}) that the recurrent configurations
correspond to the parking functions which play a central role in combinatorics.
We obtain a simple algorithm to compute the rank in that case, of linear arithmetic complexity,  while 
there is no known polynomial time  algorithm to compute that rank for arbitrary graphs (see \cite{bakerShokrieh} \cite{manjunath}).
This algorithm suggests the introduction of a new parameter on parking functions
and Dyck paths which we call prerank.  For this parameter we study some properties in relation with other known parameters on Dyck paths.
The degree is the other statistic on configuration involved by Baker and Norine theorem.
The restriction to sorted parking configurations select exactly one representative of each possible run of the algorithm up to the symmetries of the complete graph with a pointed vertex. 
We prove that the generating function of sorted parking configurations according to their degree, their rank and the size of complete graph, is, up to a change of variable and a rescaling, a symmetric fraction in two copies of Carlitz q-analogue of Catalan's numbers.

\section{Configurations on a graph}

Let $G= (X, E)$ be a multi-graph with $n$ vertices, where $X =\{x_1,x_2, \ldots, x_n\}$
is the vertex set and $E$ is a symmetric matrix such that $e_{i,j}$
is the number of edges with endpoints $x_i, x_j$, hence
$e_{i,j} = e_{j,i}$. In all this paper  $n$ denotes
the number of vertices of the graph $G$ and $m$ the number of its edges.
We  suppose that $G$ is connected and 
has no loops, so that $e_{i,i}=0$ for all $i$.

We will consider {\em configurations} on a graph,  these  are  elements of 
the discrete lattice $\Z^n$.  Each configuration $f$ may be considered as 
assigning (positive or negative) tokens to the vertices.
The symbol $\eps^{(i)}$ 
will   denote the configuration in which the value 1
is assigned to vertex $x_i$ and the value 0 is assigned to all other vertices.

The  {\em degree} of the configuration $f$ is the sum
of the $f_i$'s it  is denoted by  $deg(f)$. 

\subsection{The Laplacian configurations}
These configurations correspond to the  rows of the Laplacian 
matrix of a graph, a classical tool in Algebraic Graph Theory.

The  {\em Laplacian configuration} $\Del{i}$ is given by:
$ \Del{i} = d_i \eps^{(i)}- \sum_{i=1}^n e_{i,j}\eps^{(j)}$, 
where   $d_i = \sum_{i=1}^n e_{i,j}$ is the degree of the vertex $x_i$.
This $n$ configurations which degrees are equal to 0  play a central role throughout this paper.

We denote by  
 $L_G$ the subgroup of  $\Z^n$
 generated by the $\Del{i}$, and two configurations $f$ and $g$ will be said 
{\em toppling equivalent} if $f -g\in L_G$, which will also be written as
$ f \simil g$.

In the sandpile model, the transition from 
configuration $f$  to  the configuration $f- \Del{i}$ is allowed only if $f_i \geq d_i$ and is called a toppling, 
it is called a firing in the theory of chip firing games, here we omit this condition and perform topplings even
if $f_i < d_i$.

Notice that $\sum_{i=1}^n \Del{i} = 0$ and that for a connected graph 
this is the unique relation (up to multiplication by a constant) satisfied by the $\Del{i}$, moreover 
the principal minors of the Laplacian matrix are all equal to the number of spanning trees
of the graph.

\subsection{Recurrent configurations}
We use here the notation usually considered in the sandpile model, so that we will call
{\em sandpile configuration} a configuration $f$ such that $f_i \geq 0$ for all
$i < n$. This corresponds to the fact that in the sandpile model the vertex $x_n$ is 
considered as a sink collecting tokens, so that  the number of tokens of the sink  is not considered
in this context. 

\begin{definition}
\label{def:toppling}
In the sandpile model, a toppling on vertex $x_i$, where $i \neq n$, may occur in 
a sandpile configuration only if if $f_i \geq d_i$. A sandpile configuration $f$ is stable if no toppling can occur, that is
$f_i < d_i$ for all $i <n$.
\end{definition}
Notice that when a toppling occurs in the sandpile model, the configuration 
 $f - \Del{i}$ is  also a sandpile configuration.

The toppling  operation for a sandpile configuration will be denoted
by $ f \rightarrow g$.
We also write:
$$ f \etoile g$$ if $f$ and $g$ are sandpile configurations 
and if $g$ is obtained from $f$ by a sequence 
of  toppling operations meeting only sandpile configurations.  Notice that $ f \etoile g$ implies $ f \simil g$.

\medskip

Sequences of topplings may be performed in any order
until a stable configuration is attained as the following proposition 
states, the proof of which may be found in \cite{dhar1} or in \cite{knuthVol4} pages 42, and 70.

\begin{propo}
\label{prop:uniqueStable}
For any sandpile configuration $f$ there exists a unique
stable configuration  $f'$  such that $f \etoile f'$.
\end{propo}

\medskip
A configuration is {\em recurrent} in an evolving system if it could be
observed after a long period of the evolution of the system. In the
case of the  sandpile model, the system is considered to evolve
by adding a token in any cell at random and then applying 
topplings  until a stable configuration is reached. This
translates into the  following notion which is  central~:

\begin{definition}
\label{def:recurrent}
A configuration $f$ is recurrent if it is stable and
 there exists a sandpile  configuration $g \neq 0$ such that $f +g \vufort
f$.
\end{definition}

\medskip

\noindent
The following  important result, giving canonical representatives in the classes of the 
relation  $\simil$  is obtained in \cite{dhar2, biggs1, coriRossin} by different ways.

\begin{theorem}
For any configuration $f$ there exists a unique recurrent configuration $g$ such that
$ f \simil g$.
\end{theorem}

\medskip

In order to characterize the recurrent configurations D. Dhar
used  the configuration $\Del{n}$ and proposed the following algorithm.

\begin{theorem}{\rm Burning Algorithm.}
\label{cor:burning}
The stable  configuration $u$ is recurrent if and only if
$$f - \Del{n} \vufort f$$
Moreover in this sequence of topplings each vertex different from  $x_n$ topples exactly
once.
\end{theorem}

\noindent
This algorithm can be translated into  another characterization, giving:

\begin{coro}
\label{cor:caracRecur}
A stable configuration $f$ is recurrent if and only if for any non-empty subset
$Y$ of $X \setminus \{x_n\}$ there is at least  an $x_k$ in $Y$ such that
its  degree in the subgraph spanned by $Y$ is greater than 
or equal to $f_k$, more precisely if the following condition is satisfied:
\begin{equation}
\label{eq:antiParking}
f_k \geq  \sum_{x_i\in Y} e_{i,k} 
\end{equation}
\end{coro}

\begin{proof}Let $f$ be a recurrent configuration, and $Y$ be a subset of $X$,  then by Dhar's Burning Algorithm,
starting from the configuration $f- \Del{n}$ 
there is a sequence of topplings 
 of the vertices  in which any vertex topples. We may 
suppose that the vertices are numbered 
in the order in which they topple, $x_1$ just after $x_n$ ,
then $x_2$ and so on until
$x_{n-1}$ then for allowing a toppling at
 vertex $x_i$ each $f_i$ has to satisfies the condition:
$$ d_i \leq f_i + \sum_{j=1}^{i-1}  e_{i,j}$$ 
Now for any subset   $Y$ of $X$,
let $k$ be the smallest integer such that $x_k \in Y$, then since there
is no $x_i\in Y$ with $i$  less than $k$ we have:
$$ d_k \geq  \sum_{j=1}^{k-1}  e_{j,k} +   \sum_{x_i\in Y} e_{i,k}$$
Putting  $i=k$ in the first inequality and the two inequalities together gives the result.

\medskip
Conversely if $f$ is a stable configuration satisfying 
condition \ref{eq:antiParking}  we build a toppling sequence 
starting with vertex $x_n$, then taking as $x_1$ the vertex in $Y=\{x_1, x_2, \ldots, x_{n-1}\}$ satisfying $f_1 \geq \sum_{i=2}^{n-1} e_{1,i}$, this vertex can 
topple after $x_n$ 
 since in that case $f_1 + e_{1,n} \geq \sum_{i=2}^{n} e_{1,i}= d_1$.
Then at each step, a  vertex $x_j$ such that
  $f_j \geq \sum_{i=j+1}^{n-1} e_{i,j}$ exists taking $Y = X \setminus
\{x_n, x_1, x_2, \ldots x_{j-1}\}$, this vertex can topple at this stage.
We have thus built a sequence of toplings proving that $f$ is recurrent.
\end{proof}

\subsection{Parking configurations}
\label{sec:ParkingConf}
We consider a kind of dual notion to that of  recurrent configuration,
  such configurations are often called
 {\em parking configurations} since in the case of complete
graphs, these are exactly the parking functions, a  central object in combinatorics.

\begin{definition}
\label{def:parking}
A sandpile configuration $f$ on a graph $G$ is a  parking configuration if
for any subset $Y$ of
$X\setminus \{x_n\}$ there is a vertex $x_k$ in $Y$ such that 
$f_k$ is less than the number of edges which are incident to $x_k$ and a vertex out of $Y$.
 More precisely if the exists  $x_k \in Y$  such that 
$ f_k<  \sum_{x_i \notin Y} e_{i,k}$.
\end{definition}

In other words  a sandpile configuration $f$ 
is a  parking configuration if and only if 
 there is no 
toppling of all the vertices in a subset $Y$ of $\{x_1, x_2, \ldots x_{n-1}\}$ leaving all the  $f_i \geq 0$.
For this reason these configurations are also called {\em superstable} (as for instance in \cite{perkPerlWilm}).

\begin{propo}
\label{prop:recToParking}
Let $f$ be a stable configuration and let $\delta $ be the configuration such 
that $\delta_i = d_i - 1$. Define $\beta(f) =  \delta- f$. Then $f$ is recurrent if and only if $\beta(f)$
is a parking configuration.
\end{propo}
\begin{proof}
It suffices to compare Corollary \ref{cor:caracRecur} and Definition  \ref{def:parking}
and to notice that:
$$ d_k = \sum_{x_j\notin Y} e_{k,j} + \sum_{x_j\in Y} e_{k,j}$$

\end{proof}
\begin{coro}
\label{prop:caracParking}
For any configuration $f$ there exists a unique parking configuration
$g$  such that $f \simil g$.
\end{coro}

\begin{proof}
For any  configuration $f$  let $g$ be the recurrent configuration such that 
$g \simil \delta - f $ then $ \beta(g)$ is a parking configuration such that $ f \simil  \beta(g)$.
\end{proof}

In this paper we will often consider the parking configuration  in a class as a representative of this class.
A parking configuration $f$  in a graph with $n$ vertices will be represented by the subsequence  consisting of this first $n-1$ terms
and an integer $s$  such that:

\begin{equation}
\label{eq:defParking}
(f_1, f_2, \ldots , f_{n-1}) {\hskip 1cm}   s = f_n
\end{equation}

\noindent 
hence $s$ represents the number of tokens on the distinguished often called  " sink " vertex $x_n$.

\subsubsection*{Parking configurations and acyclic orientations}

An orientation of $G$ is a directed graph obtained from $G$ by
orienting each edge, so that one end vertex becomes  the head and the
other one  the tail.  A directed path in such a graph
consists of a sequence of edges such that the head of an edge is
equal to the tail of the subsequent one.

The orientation is acyclic if
there is no directed circuit, i.e. a directed 
path starting and ending at the same
vertex. We associate to any parking configuration $f$
an acyclic orientation by :

\begin{propo}
\label{prop:parkingToAcyclic}
For any parking configuration $f$ on $G=(X,E)$ there exists
at least one  acyclic orientation $\overrightarrow{G}$ such that for any vertex $x_i$,
 $i\neq n$, $f_i$  is strictly less  than its 
indegree $d^{-}_i$ (i.e. the number of edges with head $x_i$).
\end{propo}

\begin{proof}
We orient the edges using an algorithm that terminates after $n$ steps.
Consider $Y = \{x_1,x_2, \ldots, x_{n-1}\}$. From the definition of 
 parking configurations, 
there is at least one vertex $x_i$ such that 
$f_i < e_{i,n}$ then orient all  these $e_{i,n}$
 edges from $x_n$ (the tail) to $x_i$ (the head), and remove $x_i$ from $Y$.
Repeat the following operation until $Y$ is empty:
   \begin{itemize}
\item Find $x_k$ in $Y$ such that $f_k < \sum_{x_j \notin Y} e_{k,j}$; orient all the edges joining any vertex $j$ 
outside $Y$ to $x_k$  from $x_j$ to $x_k$ and remove $x_k$ from $Y$.
\end{itemize}

\end{proof}
In the preceding proof one may recognize a scheduling of topplings related to the Dhar criterion applied to the recurrent configuration $\beta(f)$.
Notice that more precise results involving maximal parking configurations are given in \cite{bensonChakTetali}
\section{Configurations on the complete graph}

\subsection{Configuration classes in $K_n$}
\label{subsec:algo1}
In the complete graph $K_n$ each of the $n$ vertices has all the $n-1$ other vertices as neighbors via a simple edge.

 For  a configuration $f$ in the complete graph $K_n$  the determination of the parking configuration equivalent to it is facilitated by the 
following lemma and its corollaries:

 \begin{lemma}
 \label{lem:equiv0}
 A configuration $f$ of $K_n$ is toppling equivalent to 0 if and only if the two following conditions are satisfied:
$$
 deg(f) = 0  \hskip 0.5cm \makebox{\rm and}  \hskip 0.5cm
 \forall i, j   \leq n  \hskip 0.2cm f_i = f_j  \hskip 0.5cm (\bmod  \ n)$$
 \end{lemma}
 \begin{proof}
 The necessary condition follows from  
the fact that these relations  are not modified by any toppling and are satisfied by the
parking configuration equivalent to $0$ which is  $(0,  \ldots , 0)$ and by all the $\Del{i}$.
Ä

\bigskip

The sufficient condition is obtained by induction on the sum of the $|f_i|$. 
More precisely,  for a given 
configuration $f$ satisfying the conditions above,
let $i$ be such that $|f_i|$ is maximal. Replacing $f$ by $-f$ if necessary, allows to assume  $f_i > 0$.

We consider two cases:
\begin{itemize}
\item If $f_i > n$ then, since $deg(f) = 0$, there exists $j$ such that $f_j <0$. Let us  consider 
$g=f-\Del{i} + \Del{j}$. We have $g_k = f_k$ for $k \notin \{i,h\}$ and $g_i = f_i -n,g_j = f_j+n$  hence
 this configuration satisfies also  the conditions of the Lemma;  and $\sum_i |g_i| < \sum_i |f_i| $ which allows us  to use the inductive hypothesis and obtain $ g \simil 0$ and 
hence $f = g+ \Delta^{(i)}-\Del{j} \simil 0$.

\item If $0 < f_i \leq n$, the equations above imply that the only possible values for $f_j$ are either $f_i$ or $f_i-n$.
Let $k \geq 1$ denote the number of $j$ such that  $f_j=f_i$. 
Since $deg(f) = 0$ we have:
$$kf_i + (n-k) (f_i-n) =0$$
hence $k=n-f_i$.
Let $ I = \{j | f_j = f_i \}$ then it is then  not difficult to check that:
$$ f = \sum_{j \in I} \Delta^{(j)}.$$
\end{itemize}

 \end{proof}
 
 Notice that this characterization may also be considered as given by  the toppling invariants defined in \cite[equation(3.11)]{dharRuelleVerma}.

 \bigskip
 
 Using the fact that $ f \simil g$ if and only if $ f -g \simil 0$ we obtain:
 
\begin{coro}
\label{coro:equivinKn}
Two configurations $f$ and $g$ are toppling equivalent in $K_n$
if and only if the following holds
 $$deg(f) = deg(g)  \hskip 0.7cm  \makebox{\rm and for any} \hskip 0.2cm1 \leq i, j \leq n,  
\hskip 0.2cm f_i - f_j= g_i -g_j \hskip 0.2cm (\bmod  \ n).$$
\end{coro}

Another consequence is the following:
\begin{coro}
\label{coro:nbSandPByClass}
Any class of the toppling equivalence on $K_n$ contains exactly $n$ sandpile configurations
such that $0 \leq f_i < n$ for all $i < n$, all the stable configurations of the class  are among them.

\end{coro}

\begin{proof}
For each class of the toppling equivalence there is a parking configuration $f$ which 
clearly satisfies the condition. 
Now for any $k = 1, 2, \ldots , n-1$ the configuration $f^{(k)}$ given by 
$$f^{(k)}_i = f_i + k {\hskip 0.3cm} (mod \ n)$$
for $i < n$ and $f^{(k)}_n = deg(f) - \sum_{i=1}^{n-1} f^{(k)}_i$, 
is a configuration equivalent to $f$ and all these configurations are distinct.

Conversely a configuration $g$ satisfying $g_i \leq n-1$ for all $i <n$ equivalent to $f$ is  
equal to one of the $f^{(k)}$, where $k = f_1 - g_1{\hskip 0.2cm} (mod \ n)$
\end{proof}

\bigskip

As a direct consequence of Corollary \ref{coro:equivinKn} we have an efficient 
algorithm computing from any configuration of $K_n$ an equivalent configuration
with relatively small values of the $f_i$ for $i < n$:

\noindent
{\bf Algorithm.} Given a configuration $f$ of $K_n$  one can find a configuration $g$ toppling equivalent
 to $f$ and such that $0 \leq g_i < n$ for any $1 \leq i \leq n-1$ by setting:
 
  $$g_i= f_i -f_1 \ (\bmod \ n) {\hskip 0.15cm}  \mbox{\rm for} {\hskip 0.1cm}  i < n,  {\hskip 0.3cm}    \mbox{\rm and} {\hskip 0.3cm}    g_n = deg(f) - \sum_{i=1}^{n-1} g_i.$$
  
  \bigskip

Notice that this algorithm performs $O(n)$ arithmetic operations on integers.

  \subsection{Parking configurations}
  
 Since in  the complete graph $K_n$ the $n$ vertices have all the $n-1$ other vertices as neighbors,
 a configuration $f$ is a parking configuration if for any subset $Y$ of $\{x_1, x_2 \cdots, x_{n-1} \}$ containing $p$ vertices 
there is 
at least one vertex $x_k \in Y $ such that $f_k < n-p$.  
The sequences $f_1, f_2, \ldots , f_{n-1}$ satisfying this condition 
are well known in combinatorics and are called  parking functions, they are characterized by the simpler
following condition:
\begin{propo}
\label{prop:parkingFunction}
A sequence $(f_1, f_2, \ldots, f_{n-1})$ corresponds to the first $n-1$ values
of a parking configuration of $K_n$, if and only if after reordering it as a weakly increasing sequence
$(g_1, g_2, \ldots, g_{n-1})$ one has $g_i < i$, for all $i= 1, \ldots,  n-1$.
\end{propo}

\subsection{Dyck words} 
\label{sec:dyck}

We consider words on  the alphabet with two letters $\{a,b\}$. For a word $w$,
we denote $|w|_x$  (where $ x \in \{a,b\}$) the number of occurrences of the letter $x$ in  the word $w$.
Hence the {\em length} $|w|$ of $w$ is equal to $|w|_a+|w|_b$, we use also the mapping $\delta$ associating to any word $w$ the integer $\delta(w) = |w|_a- |w|_b$.
The word $u$ is a  {\em prefix} of $w$ if $w = uv$, this prefix is  {\em strict} if $u \neq w$. 
Two words $w, w'$ are {\em conjugate}, if they may be written like $w=uv, w' =vu$.
 For any positive integer $n$, we denote $A_n$ the set of words containing $n-1$ occurrences of the letter  $a$ and $n$ occurrences of $b$.
Notice that a word $w$ is in  $A_n$ if and only if $|w| = 2n-1$ and  $\delta(w) = -1$.
The set  $D_n$ of Dyck words followed by an occurrence of $b$ will have a central role in 
the study of configurations on $K_n$, it is defined as follows,
\begin{definition}
\label{def:Dn}
The subset  $D_n$ of  $A_n$ is the set of words $w$  such that $\delta(u) \geq 0$ for any strict  prefix $u$ of $w$.
\end{definition}
A {\em Dyck word} is a word $w$ such that $|w|_a = |w|_b$ and $\delta(u) \geq 0$ for any 
prefix $u$ of $w$, hence $w$ is a Dyck word if and only if $wb \in D_n$ for $n= |w_a|+1$.

The classical Cyclic Lemma (see  \cite{dvoMotzk}) may be stated as follows:

 \begin{lemma}
 \label{lem:cyclic} 
 Any word $w$ of $A_n$  admits a unique
 factorization 
 $$w  = uv$$ 
 such that $vu$ is in $D_n$.
 Moreover  $u$ is the shortest prefix of $w$ such that $\delta(u)$ is minimal among all 
 prefixes of $w$.
  \end{lemma}

We will often consider configurations $f$ satisfying  the following condition:
\begin{equation}
\label{eqn:prePark}
0  \leq f_1 \leq f_2 \leq \ldots f_{n-1} \leq n
\end{equation}

\begin{definition}
\label{def:phi1}
To any configuration satisfying condition (\ref{eqn:prePark})
 we associate  
 the word  $ \phi_1(f) $ of $A_n$
   such that the prefix of $\phi_1(f)$ ending with the $i$-th occurrence of the letter $a$ contains exactly $f_i$ occurrences of the letter $b$. 
  \end{definition}
  
  \begin{propo}
  \label{prop:conjugue}
  Let $f,g$ be two configurations of $K_n$  such that their first $n-1$ values satisfy condition  (\ref{eqn:prePark}).
  If  $deg(f) = deg(g),  \phi_1(f) =  uv,  \phi_1(g) = vu$ then:
  $$g \simil f'= (f_{p+1}, \cdots ,f_{n-1}, f_1, \cdots , f_p, f_n)$$
  where $p = |u|_a$.
  \end{propo}
  \begin{proof}
  Let $q$ be the number of occurrences of $b$ in $u$, and $p' = n-1-p$.
 Since $\phi_1(f) =uv$ and $\phi_1(g) = vu$ we have:
\[
g_i = \left\{ \begin{array}{ll}
f_{p+i} -q & \mbox{if $0< i < n-p$} \\
  f_{i-p'} +n-q& \mbox{if  $n-p \leq i <n$}
        \end{array}
        \right. 
        \]
 This shows that $ g_i - f'_{i} = -q$ if $i < n-p$ and $g_{i} -f'_i = n-q$, if $ i \geq n-p$ but all these are equal $\mod n$,
and  the result follows from Corollary \ref{coro:equivinKn}, since $deg(f') = deg(f)=deg(g).$
        
   \end{proof}   
   
   \begin{remark}  
   \label{rem:calculConjugue}
   Let $f,g,u,v,p,q$ be as in the Proposition above then:
   $$g_n=f_n + n(q-p)-q.$$
  \end{remark}
\begin{proof}
  The first $n-1$ values of  $f$ and $g$  are such that  $n-p-1$ values pf $f$  are decreased by $q$ and $p$ others
 are increased by $n-q$ giving:
$ \sum_{i=1}^{n-1} g_i = \sum_{i=1}^{n-1} f_i  -q(n-p-1) +(n-q)p,$
hence computing $deg(g)$ we have:
$$ deg(g) = g_n+ \sum_{i=1}^{n-1} g_i = g_n + \sum_{i=1}^{n-1} f_i  +n(p-q) +q$$
The result follows from $deg(g) = deg(f) = f_n +\sum_{i=1}^{n-1} f_i $.
\end{proof}

  We introduce two mappings $\phi$ and $\psi$, on configurations  of $K_n$ the first associates to any  configuration a word in $D_n$,
 and  the second one associates to it an integer, these two values will give many informations on the configuration.

  \begin{definition}
  \label{def:phi}
  Let  $f$ be any configuration on $K_n$, $g$ the parking configuration such that $f \simil g$  and $g'$ the
  configuration obtained from $g$ by reordering the $n-1$ first terms in weakly increasing order, then the 
  mapping $\phi$ and $\psi$ are defined by:
  
  $$ \phi(f) = \phi_1(g'), \hskip 1cm \psi(f) = g_n.$$
   \end{definition}
  Notice that from the definition of parking configurations, for any $i$ the $i$-th  occurrence of $a$ in $\phi(f)$
  is preceded by less than $i$ occurrences of $b$, showing that $\phi(f)$ is a word in $D_n$.
  Since there is a unique parking configuration in any class of $\simil$ we have that  $f \simil g$ implies $\phi(f) = \phi(g)$ and $\psi(f) = \psi(g)$.
   \begin{lemma}
   \label{lem:permutPhi}
  Let $f,g$ be two configurations such that the first $n-1$ entries of $g$ are obtained by permuting those of $f$,
  then $\phi(f)=\phi(g)$
  \end{lemma}
  \begin{proof}
  Let $\alpha$ be the permutation such that $f_i = g_{\alpha(i)}$ for $1 \leq i \leq n-1$ and let 
  $f- \sum_{i=1}^{n-1} a_i \Delta^{(i)}$ be the parking configuration equivalent to $f$, then
  $g- \sum_{i=1}^{n-1} a_{\alpha(i )}\Delta^{(\alpha(i))}$ is the parking configuration equivalent to $g$, showing
  that these two parking configurations differ only by a permutation of their first $n-1$ entries. Hence 
  their sorted versions are equal ending the proof.
  
  \end{proof}
  
 \begin{propo}
 \label{prop:uniqConj}
 For any configuration $f$ satisfying condition  (\ref{eqn:prePark}), the word $\phi(f)$ is the unique conjugate 
 of $\phi_1(f)$ which is an element of $D_n$.
 \end{propo}
  
 \begin{proof}
 Let $\phi_1(f)=uv$ be such that $vu$ is an element of $D_n$, and let $g$ be the configuration such that
 $deg(g) = deg(f)$ and $\phi_1(g) = vu$. Clearly $g$ is a parking configuration and is by Proposition \ref{prop:conjugue}
 such that a permutation
 of it is $\simil$ equivalent to $f$. By the preceding Lemma we have that $\phi(f) = \phi(g)= \phi_1(g)$.
 \end{proof}

\subsubsection*{A linear time algorithm to compute the parking configuration equivalent to a given configuration}
The preceding proof gives  also an algorithm finding the recurrent configuration of any configuration,
 this algorithm performs $O(n)$ arithmetic operations on integers.
Starting from any configuration $f$, find the equivalent configuration $f'$ which first $n-1$ values are 
non negative and  not greater than $n$, as described at the end of section~\ref{subsec:algo1}. 

Then  sort the first $n-1$ first entries of $f'$ 
(this can be done in linear time, since the sorted values are between $0$ and $n-1$) to obtain $f"$. 
 Let  $w =\phi_1(f")$ and determine  the decomposition $w=uv$ by identifying in linear time the shortest prefix $u$ of $w$  such that $\delta(u)$ is minimal.
Determine $q = |u|_b$ then  deduce the parking configuration $g$  by :

\[
g_i = \left\{ \begin{array}{ll}
f'_{i} -q & \mbox{if $f'_i \geq q$} \\
  f'_{i} +n-q& \mbox{if $f'_i < q$}
        \end{array}
        \right. 
        \]
        
   \medskip

We extract from the end of this algorithm and the proof of Proposition~\ref{prop:uniqConj} a corollary on the 
sorted parking configuration toppling equivalent to a configuration $f$.

We will say that two configurations $f$ and $g$ of $K_n$ are toppling equivalent after permutation,   if there exists a permutation $\sigma\in S_{n-1}$ such that $f'=\sigma(f) =(f_{\sigma(1)},\cdots, \,f_{\sigma(n-1)},f_n)$ is such that  $f'\simil g$.

\begin{coro}
\label{cor:calcParkingPerm}
Let $w$ be any word of $A_n$, and 
let $w=uv$  be the factorisation of $w$ given by the cyclic lemma (Lemma~\ref{lem:cyclic}).
Then two configurations $f$ and $g$ satisfying equation (\ref{eqn:prePark})  and the equalities:
$$\phi_1(f) = uv, \hskip 0.5cm \phi_1(g)= vu ,  \hskip 0.5cmdeg(f)=deg(g)$$
 are such 
that  $g$ is a parking configuration which is toppling equivalent  after permutation to the configuration $f$.
\end{coro}
   
\section{Effective configurations}
We come back in this section with general graphs $G$ not necessarily equal to a complete graph, define the 
notion of $L_G$-effective configuration and recall the main results of \cite{bakerNorine},  the proofs we give in this section are more 
or less a reformulation in our terms  of the proofs of them given in \cite{bakerNorine}.
The game described in the introduction can be translated 
in determining if a configuration is $L_G$-effective with the following definition of effectiveness:
\begin{definition}
\label{def:effective}
A configuration $f$ is effective if 
$f_i \geq 0 $ for all $i$.
A configuration $f$  is $L_G$-effective if there exists an effective 
configuration $g$ toppling equivalent to $f$ (recall that this means $f -g \in L_G$). 
\end{definition}

Since two   equivalent configurations by $\simil$  have the same degree,
it is clear that a configuration with negative degree
is not $L_G$-effective. However we will prove that  configurations
with positive degree are not necessarily $L_G$-effective as these two 
examples show:

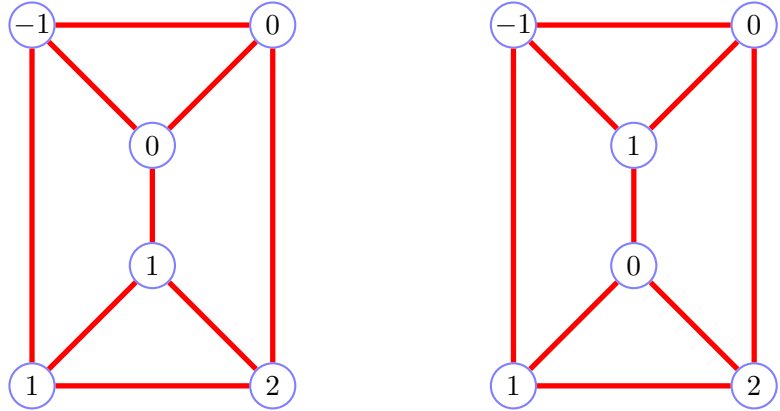
\begin{figure}[H]
\begin{center}
\begin{tikzpicture}[scale =0.8]

\tikzstyle{sommet}=[circle,draw=blue!50,thick,
                   inner sep=0pt,minimum size=6mm]
\foreach \x/\y/\name/\value in {0/0/q0/1,2/2/q1/1,4/0/q2/2,2/4/q3/0,0/6/q4/{\footnotesize -1},4/6/q5/0}{
\node[sommet] (\name) at (\x,\y) {\small $\value$};
}
\foreach \source/\dest in {q0/q1,q2/q0,q2/q1,q3/q1,q4/q0,q4/q3,q4/q5,q5/q3,q5/q2}{
\draw[draw=red,line width=2] (\source) to (\dest);
}

\foreach \x/\y/\name/\value in {0/0/q0/1,2/2/q1/0,4/0/q2/2,2/4/q3/1,0/6/q4/-1,4/6/q5/0}{
\node[sommet] (\name) at (8 +\x,\y) {\small $\value$};
}
\foreach \source/\dest in {q0/q1,q2/q0,q2/q1,q3/q1,q4/q0,q4/q3,q4/q5,q5/q3,q5/q2}{
\draw[draw=red,line width=2] (\source) to (\dest);
}
\end{tikzpicture}

\caption{An $L_G$-effective configuration and a non $L_G$-effective one}
\label{fig:effectConfig}

\end{center}
\end{figure} 

\subsection{Configuration associated to an acyclic orientation of $G$}

As already seen in Section \ref{sec:ParkingConf} an orientation of $G$ is a directed graph obtained from $G$ by orienting
each edge, that is distinguishing for each edge with end points $x_i$ and
 $x_j$ which one is the head and the other being the tail.
 The orientation is acyclic if
there is no directed circuit.  Let $\overrightarrow{G}$ be an acyclic orientation of $G$, we define
the configuration $f_{\overrightarrow{G}}$ by: 
$$ (f_{\overrightarrow{G}})_i = d_i^{-} -1$$
Where $ d_i^{-}$ is the number of edges  which have  head  $x_i$.
 The configuration represented in the right of 
  Figure \ref{fig:effectConfig} is equal to $f_{\overrightarrow{G}}$ for the orientation 
of $G$ represented  in  Figure \ref{fig:acycOrient}.

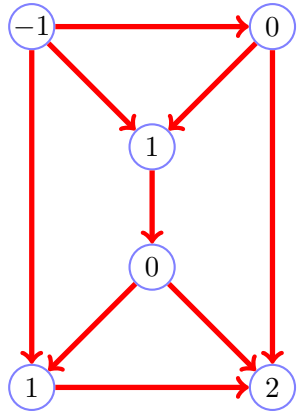
\begin{figure}[H]

\begin{center}
\begin{tikzpicture}[scale =0.8]

\tikzstyle{sommet}=[circle,draw=blue!50,thick,
                   inner sep=0pt,minimum size=6mm]
\foreach \x/\y/\name/\value in {0/0/q0/1,
2/2/q1/0,
4/0/q2/2,
2/4/q3/1,
0/6/q4/-1,
4/6/q5/0}{
\node[sommet] (\name) at (\x,\y) {\small $\value$};
}
\foreach \source/\dest in {q0/q2,q1/q0,q1/q2,q3/q1,q4/q0,q4/q3,q4/q5,q5/q3,q5/q2}{
\draw[draw=red,->,line width=2] (\source) to (\dest);
}
\end{tikzpicture}

\caption{An orientation of $G$ and the corresponding  configuration}
\label{fig:acycOrient}
\end{center}
\end{figure}

\begin{propo}
\label{propo:acyclicNonEffective}
The configuration associated to an acyclic orientation of $G$ is 
not $L_G$-effective.
\end{propo}

\begin{proof}
We will show that for any linear combination $g = \sum_{i=1}^n a_i \Del{i}$
the sum $h$ of $g$ and $f_{\overrightarrow{G}}$ is not an effective   configuration.
Let $\varepsilon_{i,j}$ denote the number of edges with head $x_j$ and tail
$x_i$. Then $e_{i,j} = \varepsilon_{i,j} + \varepsilon_{j,i}$ (but notice
that since the orientation is acyclic, at least one of the two values in the sum above is 
equal to 0).

For any vertex $x_i$ of $G$ we have $d_i^{-} = \sum_{j=1}^n \varepsilon_{j,i}$
so  that: 
$$ \eta_i = -1 + \sum_{j=1}^n \varepsilon_{j,i} + a_id_i - \sum_{j=1}^na_je_{i,j}$$
Using $d_i = \sum_{j=1}^n e_{j,i}$ and decomposing each $e_{i,j}$ into
$\varepsilon_{i,j} + \varepsilon_{j,i}$ gives;

\begin{equation}
\eta_i = -1 + \sum_{j=1}^n \varepsilon_{j,i} + a_i\sum_{j=1}^n (\varepsilon_{i,j} + \varepsilon_{j,i})  - \sum_{j=1}^na_j(\varepsilon_{i,j} + \varepsilon_{j,i})
\end{equation}

Giving:

\begin{equation}
\eta_i = -1 + \sum_{j=1}^n (1+a_i-a_j)\varepsilon_{j,i} 
+ \sum_{j=1}^n (a_i-a_j) \varepsilon_{i,j} 
\end{equation}

If there is a unique minimal value, say $a_k$ 
 among the $a_i$, that is $a_k < a_i$ for all $i \neq k$ then since the 
$a_i$ are integers $1 + a_k -a_j \leq 0 $ and  $\eta_k < 0$.

If there are many $a_i$'s attaining the minimal value take $k$ be such that
$a_k$ be among them  and $\varepsilon_{j,k} = 0$ for all the other
minima $j$, the existence of such a $k$ follows from the acyclicity of
$\overrightarrow{G}$. Then for this $k$ we have $\eta_k < 0$.

\end{proof}

\subsection{Characterisation of $L_G$-effective configurations}

The following Theorem is the central result in \cite{bakerNorine}.

\begin{theorem}
\label{th:caracEffective}
A configuration $f$ is $L_G$-effective if and only if the parking 
configuration $g$ equivalent to $f$ is such that $g_n \geq 0$.
Moreover for any configuration $f$ one and only one of the following 
assertions is satisfied:

(1) $f$ is $L_G$-effective

(2) There exists an acyclic orientation $\overrightarrow{G}$ such that
$f_{\overrightarrow{G}} - f$ is $L_G$-effective.

\end{theorem}

\begin{proof}
Let $f$ be a configuration and let $g$ be the parking configuration
in its class. If $g_n \geq 0$ then $f$  is $L_G$-effective since it is equivalent to $g$ 
which is effective.  If $g_n < 0$ 
then the acyclic orientation  $\overrightarrow{G}$ of $G$ given by Propostion
 \ref{prop:parkingToAcyclic}
is such
that the indegree $d_i^-$  of each vertex $x_i$ except $x_n$
 in $\overrightarrow{G}$
satisfies $g_i < d_i^-$, since $(f_{\overrightarrow{G}})_n \geq -1$ we have
$g \leq f_{\overrightarrow{G}}$ proving that $g$ is not $L_G$-effective, so is $f$ since  $f \simil g$.

 Let $f$ be non $L_G$-effective,  consider the parking configuration $g$ equivalent to $f$
 and the acyclic orientation given by Proposition  \ref{prop:parkingToAcyclic}, let
$h= f_{\overrightarrow{G}} -g$. Then for $i \neq n$ we have
$$ \eta_i =  d_i^- - 1 - f_i \geq 0$$
and since $g_n <0$:
$$ \eta_n = -1 + g_n \geq 0.$$
Hence  since $f$ and $g$ are  in the same class, so are 
$f_{\overrightarrow{G}} - $ and $f_{\overrightarrow{G}} - v$ showing that
$f_{\overrightarrow{G}} - u$ is $L_G$-effective.

\medskip
Notice that 
$f$ and $f_{\overrightarrow{G}} - f$ cannot be both $L_G$-effective since
their sum $f_{\overrightarrow{G}}$ would be too, contradicting 
Proposition \ref{propo:acyclicNonEffective}.

\end{proof}

\begin{coro}
\label{cor:effectiveIfLargeDegree}

Any configuration $f$ with degree greater than $m-n$ is $L_G$-effective.

\end{coro}

\begin{proof}
If $f$ such that $deg(f) > m-n $ is not $L_G$-effective, by the above 
theorem there exists an acyclic orientation $\overrightarrow{G}$ of $G$ 
such that $f_{\overrightarrow{G}} - f$ is.
But the degree of this configuration is negative, giving a contradiction.
\end{proof}

\begin{propo}
\label{polTutteNbEffective}
Let $T_G(x,y)$ be the Tutte polynomial of the graph $G$, 
and let $t_i$ be the integer coefficients given by:  $$T_G(1,y) = 
\sum_{i=0}^{m-n+1}  t_iy^i$$ 
Then the number of non equivalent $L_G$-effective configurations of degree $d$
is given by:
$$ \sum_{k = m-n+1-d}^{m-n+1} t_k$$
\end{propo}

\begin{proof}

In \cite{merinoLopez} the level of a recurrent configuration $f$ 
was defined as
$$ level(f) \  \   = \   \  \sum_{i=0}^{n-1} f_i - m +d_n$$

where $d_n$ is the degree of the vertex $x_n$.

It was proved that this level varies from $0$ to $m-n+1$ and
that the number of recurrent 
configurations of level $p$  and such that $x_n = q$ 
does not depend on $q$ and is equal to the coefficent  $t_p$
of $y^p$ in the evaluation of the Tutte polynomial $T_G(x,y)$  of $G$ for $x=1$.
A bijective proof of this result was given in \cite{coriLeBorgne1}.

Using the bijection $\beta$ defined in Proposition \ref{prop:recToParking}
we have that the number of parking configurations $g$ such that
$\sum_{i=1}^{n-1} g_i = j$ and a given value for $g_n$
 is equal to the number of recurrent
configurations $f$ such that:

$$\sum_{i=1}^{n-1} f_i = \sum_{i=1}^{n-1} (d_i - 1 - g_i)  =  \   \  2m - d_n - (n-1) -j$$
and $f_n = d_n-1 - g_n$, which 
is the number of recurrent configurations of level $k=m-n+1-j$ and a given value of $f_n$.
This number is equal to $t_k$.

In order that the configuration  $g$ of degree $d$ to be $L_G$-effective we must have $g_n \geq 0$
so that $k$ must be greater or equal to 0 and not greater than $m-n+1$,
thus ending the proof.

\end{proof}

The generating function for non-equivalent $L_G$-effective configurations according to the degree counted by the variable $y$ is $\frac{y^{m-n+1}}{1-y}T_G(1,y^{-1})$.

\section{The rank of configurations}

From now on it will be convenient to denote effective configurations 
using greek letters  $\lambda, \mu$ and configurations with no particular
assumptions on them by letters $f, g, h$.

\subsection{Definition of the rank}

\begin{definition}
\label{def:rank}
The rank $\rho(f)$ of a configuration $f$ is the integer equal to:
\begin{itemize}
\item  $-1$, if $f$ is non $L_G$-effective,

\item  or, if $f$ is  $L_G$-effective, the largest integer $r$
such that for any effective 
configuration $\lambda$ of degree $r$ the configuration $f-\lambda$ is $L_G$-effective.

\end{itemize}
\end{definition}

Denoting   {$  \mathbb{P}$} the set of effective  configurations  and 
  {$ \mathbb{E}$} the set of $L_G$-effective  configurations this definition
 can be given by the following compact formula which is valid in both 
cases:
$$ \rho(f) +1 = \min_{\displaystyle \lambda \in \mathbb{P},  f- \lambda\notin \mathbb{E} }\ \ \  deg(\lambda)$$

In other words let $f$ be a configuration of rank $\rho(f)$ and 
 {$\lambda $} be  an effective  configuration such that 
 $deg(\lambda) \leq \rho(f)$ then  {$f-\lambda$} is  $L_G$-effective; moreover    there exists
an effective configuration $\mu$ of degree $\rho(f) + 1$ such that $f-\mu$ is 
not $L_G$-effective.

An immediate consequence of this definition is that if 
$deg(f) < 0 $  or if $f =f_{\overrightarrow{G}}$
 for an acyclic orientation $\overrightarrow{G}$ then the rank of $f$ is $-1$.
 Moreover if  two configurations $f$ and $g$ are such that
$f_i \leq g_i$ for all $i$ then $\rho(f) \leq \rho(g)$.

\begin{definition}
\label{def:proofRank}
An  effective configuration $\mu$ is a proof for the rank $\rho(f) $ of an $L_G$-effective 
configuration $f$ if $f-\mu$ is not $L_G$-effective and $f-\lambda$ is $L_G$-effective for any
effective configuration $\lambda$ such that $deg(\lambda) < deg(\mu)$.
\end{definition}

Notice that if $\lambda$ is a proof for $\rho(f)$ then $\rho(f) =deg(\lambda) - 1
= deg(\lambda) + \rho(f-\lambda)$.

\begin{propo}
\label{propo:largeDegreeRank}
A configuration  $f$ of degree greater than $2m-2n$ has rank
$$ r= deg(f) - m +n -1$$

\end{propo}

\begin{proof}
We first show that for any effective configuration $\lambda$ such that $deg(\lambda) =r$, the configuration 
$f-\lambda$ is $L_G$-effective. This follows from $deg(f - \lambda) = deg(f) - r = m-n+1$
by Corollary \ref{cor:effectiveIfLargeDegree}.

\medskip
We now build a effective configuration $\lambda$ of degree $r+1$ such 
that $f-\lambda$ is not $L_G$-effective. Consider any acyclic orientation 
$\overrightarrow{G}$ of $G$ and let $g = f- f_{\overrightarrow{G}}$
then $g$ is $L_G$-effective since its degree is equal to  $deg(f) - m +n$ 
hence greater than $m-n$. Let $\lambda$ be the effective configuration
such that $ g \simil \lambda$, then $f-\lambda$ is such that 
$$f_{\overrightarrow{G}} \simil f-g \simil f - \lambda$$ 
so that $f-\lambda$ is not $L_G$-effective by Proposition \ref{propo:acyclicNonEffective}.
And the Proposition results from:
 $$deg(\lambda) = deg(g) = deg(f) - deg(f_{\overrightarrow{G}}) = deg(f) -m +n = r+1$$

\end{proof}

\subsection{Riemann-Roch like  theorem for graphs}
We give here a proof of the following theorem first proved in \cite{bakerNorine} which we estimate shorter and simpler
than the original  one.

\begin{theorem}
\label{th:RR}
 Let   {$\kappa$} be the configuration such that  {$\kappa_i = d_i -2$} for all $1 \leq i \leq n$,
 so that   {$deg(\kappa)  = 2(m-n)$}.  Any configuration   {$f$} satisfies:

     $$ \rho(f) - \rho(\kappa-f) \ \  = \ \ deg(f) + n - m$$

\end{theorem}

\begin{proof}
The main ingredient for the proof is to use Theorem \ref{th:caracEffective} and 
remark that for any acyclic orientation $\overrightarrow{G}$  the orientation $\overleftarrow{G} $ of $G$ obtained from $\overrightarrow{G}$
by reversing the orientations  of all the edges  is such that: $f_{\overrightarrow{G}} + f_{\overleftarrow{G}}
=\kappa$.

\medskip

Let $f$ be any configuration we first give an upper bound for $ \rho(\kappa-f)$, we define  $\lambda$ to 
be a proof for the rank of $f$ if $f$ is $L_G$-effective,
and to be equal to $0$  if $f$ is not $L_G$-effective. So that $\rho(f) = deg(\lambda) -1$ in both cases.

Since  $f-\lambda$  is not $L_G$-effective, we have  by Theorem \ref{th:caracEffective} that there exists 
an acyclic orientation  $\overrightarrow{G}$ of $G$ such that 
$f_{\overrightarrow{G}} - (f-\lambda)$ is $L_G$-effective, hence  equivalent to an  effective configuration
$\mu$. This may be written as:

\begin{equation}
\label{eq:proofRR1}
f_{\overrightarrow{G}} - (f-\lambda)  \simil \mu
\end{equation}

Now consider the orientation $\overleftarrow{G} $ of $G$ obtained from $\overrightarrow{G}$ 
by reversing the orientations  of all the arrows, clearly $f_{\overrightarrow{G}} + f_{\overleftarrow{G}}
=\kappa$. Hence adding $f_{\overleftarrow{G}}$ to both sides of (\ref{eq:proofRR1}) we have:
\begin{equation}
\label{eq:proofRR2}
\kappa- (f - \lambda)  \simil \mu + f_{\overleftarrow{G}}
\end{equation}
which may be written as:

$$ (\kappa - f) - \mu \simil f_{\overleftarrow{G}} -\lambda$$
Giving that $\kappa-f - \mu$ is not $L_G$-effective since the reverse of an acyclic orientation is also acyclic.
Hence by the definition of the rank we have
\begin{equation}
\label{eq:proofRR3}
\rho(\kappa-f) < deg(\mu)
\end{equation}
The degree of $\mu$ is obtained from (\ref{eq:proofRR1}) giving:
$$ deg(\mu) = deg(f_{\overrightarrow{G}}) - deg(f) + deg(\lambda) = m-n -deg(f) + \rho(f) +1$$
and:
\begin{equation}
\label{eq:proofRR4}
\rho(\kappa-f) <  m-n -deg(f) + \rho(f) +1
\end{equation}

\medskip

Now to obtain a lower bound for $\rho(\kappa-f)$ we exchange the roles of $f$ and $\kappa-f$ giving:

\begin{equation}
\label{eq:proofRR5}
\rho(f) <  m-n -deg(\kappa-u) + \rho(\kappa-u) +1
\end{equation}
Since $deg(\kappa-u) = 2(m-n) - deg(f)$, inequality (\ref{eq:proofRR5}) may be written as:

\begin{equation}
\label{eq:proofRR6}
\rho(f) + m-n -deg(f)  -1  <  \rho(\kappa-f) 
\end{equation}
Comparing inequalities (\ref{eq:proofRR4}) and (\ref{eq:proofRR6}), and noticing that the
rank is an integer gives
$$\rho(f) + m-n -deg(f)  =  \rho(\kappa-f) $$
hence proving the Theorem.
\end{proof}

\section{On the rank of configurations in the complete graph}

We are  interested here in an algorithm allowing to compute the rank 
of configurations on the complete graph $K_n$. 

\subsection{Some useful remarks on the rank}
We begin by some simple facts satisfied by the rank on any graph $G$.

\begin{lemma}
\label{lemma:augmRangDeF}
Let $f$ be a configuration on a graph $G$ and $\mu$ be an  effective configuration then:
$$\rho(f)  \leq \rho(f +\mu) \leq \rho(f) + deg(\mu)$$
\end{lemma}

\begin{proof}
It is clear from the definition of the rank that increasing the values of  the components  of  a configuration cannot
decrease the value of the rank, this proves the first part of the inequality.  For the second part,
let $\lambda$ be a proof for the rank of  $f$, then $f+ \mu - (\mu +\lambda) = f- \lambda$ is not $L_G$-effective, so that
the rank of $f+ \mu$ is strictly less  than $deg(\lambda + \mu)$ but $deg(\lambda) = \rho(f) +1$, 
giving:
$$ \rho(f + \mu)  < \rho(f) + 1 + deg(\mu)$$
which is  the expected result.
\end{proof}

\begin{coro}
\label{coro:augmRangDeUn}
Let $f$ be a configuration on a graph $G$ and $f'$ be the configuration obtained by adding
1 to one of the components of $f$; so that for one $j,  \  f'_j = f_j +1$ and for all $i \neq j,  f'_i = f_i$. Then:
$$\rho(f) \leq \rho(f') \leq \rho(f) +1$$
\end{coro}

\begin{proof}
It suffices to  apply  Lemma \ref{lemma:augmRangDeF}
to the effective configuration $\mu$ such that $\mu_j=1$ and $\mu_i= 0$
for $j \neq i$.

\end{proof}

\begin{lemma}
\label{lemma:rangDiminFP}
Let $f$ be a configuration on a graph $G$ and $\mu$ be an  effective configuration 
such that :
$$\rho(f - \mu)  = \rho(f) - deg(\mu)$$
then for each effective configuration $\mu'$ such that for all $j, \mu'_j \leq \mu_j$
we have:
$$\rho(f - \mu')  = \rho(f) - deg(\mu')$$
\end{lemma}

\begin{proof}
Since  for all $j, \mu'_j \leq \mu_j$ we can write $\mu = \mu' +\mu''$, where $\mu''$ is 
an  effective configuration.  By Lemma  \ref{lemma:augmRangDeF} we have:
\begin{equation}
\label{eq:degRank1}
\rho(f) \leq \rho(f - \mu') + deg(\mu')
\end{equation}
Applying again the same Lemma and since 
 $f -\mu' = f-\mu +\mu''$ we get:
 $$ \rho(f-\mu')  \leq \rho(f-\mu) +deg(\mu'')$$
 But it is  assumed that $ \rho(f-\mu) = \rho(f) - deg(\mu)$ giving:
 
\begin{equation}
\label{eq:degRank2}
\rho(f-\mu')  \leq \rho(f) - deg(\mu) +deg(\mu'') = \rho(f) - deg(\mu') 
\end{equation}
Comparing equations (\ref{eq:degRank1}) and (\ref{eq:degRank2}) ends the proof of the Lemma.
\end{proof}

\subsection{Main fact on the rank in $K_n$}
\begin{propo}
\label{propo:caculRangMoinsUn}
Let $f$ be a parking and   $L_G$-effective configuration on the complete graph
such that for some $i, f_i = 0$  and let  $\eps^{(i)}$ be the
configuration such that  $\eps^{(i)}_i = 1$ 
and $\eps^{(i)}_j=0 $ for $j \neq i$ then:
$$ \rho(f) = \rho(f - \eps^{(i)} )+1$$

\end{propo}

\begin{proof}
Let $\lambda$ be a proof for $\rho(f)$, since $f-\lambda$ is not $L_G$-effective
there exists at least one $j$ such that $f_j < \lambda_j$.
Consider the two configurations:
$$ g = f - \lambda_j\eps^{(j)} \mbox{ and } h = f - f_j\eps^{(j) }-(\lambda_j - f_j) \eps^{(i)}$$
these two configurations have the same components but in different
order (since $g_j = \eta_i = f_j - \lambda_j$ and $g_i= \eta_j =0$), hence by the symmetry
of $K_n$ they have the same rank. Giving:
$$ \rho(g) = \rho(h) $$
Since $\lambda$ is a proof for $\rho(f)$ we have
$\rho(f) = deg(\lambda) - 1 = deg(\lambda) + \rho(f - \lambda)$.
Hence applying Lemma \ref{lemma:rangDiminFP}  with $\mu = \lambda$ and 
$\mu' = \lambda_j\eps^{(j)}$ we obtain:

$$ \rho(f-\mu')= \rho(f) -deg(\mu') = \rho(f)-\lambda_j $$
Hence since $g=f-\mu' $, we  have also:

$$\rho(h) = \rho(g) =   \rho(f) - \lambda_j $$

Since $\lambda_j - f_j \geq 1$, we can apply again Lemma \ref{lemma:rangDiminFP} , this time 
with $\mu = f_j\eps^{(j)}+(\lambda_j - f_j) \eps^{(i)}$ and $\mu' = \eps^{(i)} $, 
giving:

$$ \rho( f -\eps^{(i)})= \rho(f) - 1$$
\end{proof}
This result does not hold for any graph, the subtraction of $1$ on the $i$-th coordinate of configuration $f$ with $f_i =0$ may leave
the rank invariant as shows the following example.
\begin{remark}
The configuration $f= (0,1,0,1,0,1)$ on the wheel graph $W_5$ given in the left of the Figure \ref{fig:wheelGraph}  has rank 0, as has
the configuration $f'= f - (0,0,1,0,0,0)$.

\end{remark}
\begin{proof} Indeed the configuration $f$  is $L_G$-effective, we first show that it has rank 0. Indeed, notice that $g = f - (0,0,0,0,1,0)$ is not $L_G$-effective  since  the acyclic orientation
given on the right part of the Figure \ref{fig:wheelGraph} gives a configuration $\eta$ such that $\eta_i \geq g_i$ for all $i$.
On the other hand the configuration $f'=(0,1,-1,1,0,1)$ is toppling equivalent to $(1,2,0,2,1,-4)$ via a toppling of $x_6$ and then to $(0,0,2,0,0,0)$ via topplings of $x_1$,$x_2$, $x_4$ and $x_5$,
hence it is $L_G$-effective and has also rank equal to 0.
\end{proof}
\tikzstyle{sommet}=[circle,draw=blue!50,thick,
                   inner sep=0pt,minimum size=6mm]

\begin{figure}
\begin{tabular}{cc}

\begin{tikzpicture}
        \node (cc) at (0,0)  [sommet] [label = below:$x_6$]{1};       
         
        \node (x1) at (-2.853,0.927)[label = left:$x_1$] [sommet]{0};
        \node (x2) at (-1.764,-2.427)[label = left:$x_2$]  [sommet]{1};
        \node (x4) at (2.853,0.927) [label = right:$x_4$][sommet]{1};
        \node (x3) at (1.764,-2.427) [label = right:$x_3$][sommet]{0};
         \node (x5) at (0,3)  [label = above:$x_5$] [sommet]{0};
         \foreach \from/\to in {x1/cc,x2/cc,x3/cc,x4/cc,x5/cc,x1/x2,x2/x3,x3/x4,x4/x5,x1/x5}
         \draw [-] [very thick](\from) -- (\to);
\end{tikzpicture}

&

\begin{tikzpicture}
        \node (cc) at (0,0)  [sommet] [label = below:$x_6$]{1};       
         
        \node (x1) at (-2.853,0.927)[label = left:$x_1$] [sommet]{0};
        \node (x2) at (-1.764,-2.427)[label = left:$x_2$]  [sommet]{1};
        \node (x4) at ( 2.853,0.927) [label = right:$x_4$][sommet]{1};
        \node (x3) at (1.764, -2.427) [label = right:$x_3$][sommet]{2};
         \node (x5) at (0,3)  [label = above:$x_5$] [sommet]{-1};
         
         \foreach \from/\to in {cc/x1,x2/cc,x3/cc,x4/cc,cc/x5,x2/x1,x3/x2,x3/x4,x4/x5,x1/x5}
         \draw [<-] [very thick](\from) -- (\to);
\end{tikzpicture}
\end{tabular}
\caption{Two configurations on the Wheel Graph with ranks 0 and -1}
\label{fig:wheelGraph}
\end{figure}
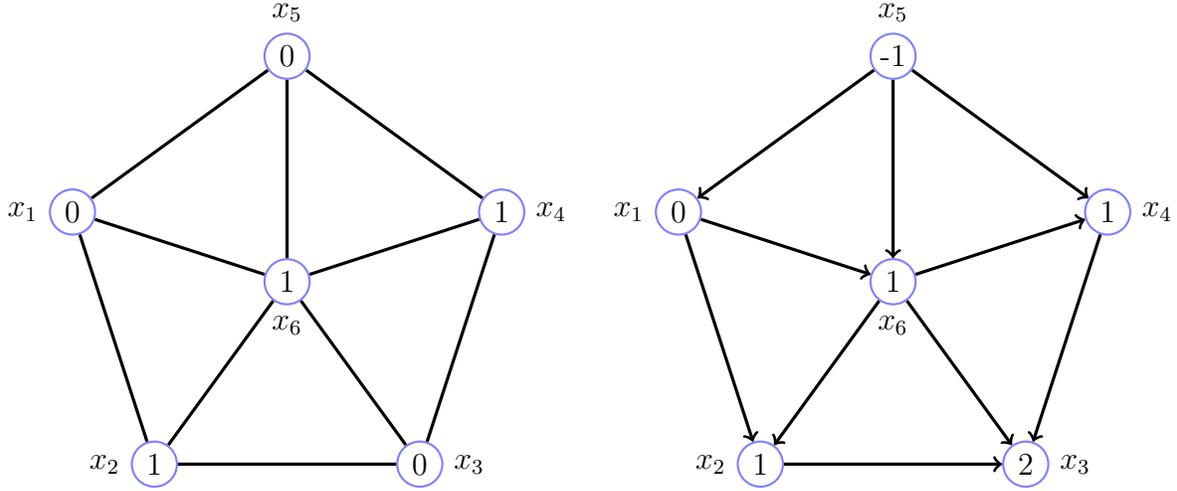

\begin{propo}
\label{propo:caculRangPourLaMax}
Let $f$ be a configuration on the complete graph such that there exists a permutation $\alpha \in \Sn{n}$
satisfying $f_{\alpha(i)} = i-1$ for $i = 1, 2, \ldots, n-1$ then :
\[
\rho(f)  = \left\{ \begin{array}{ll}
                        -1 & \mbox{if $f_{\alpha(n)} < 0$} \\
                         f_{\alpha(n) }&  \mbox{otherwise}
                                       \end{array}
                        \right. \]

\end{propo}
\begin{proof}
Since the automorphism group of $K_n$  is $\Sn{n}$ we may suppose that 
the configuration $f$ is equal to $ (0,1,2, \ldots , n-2, a)$ , where $a = f_{\alpha(n)}$.
The configuration is a parking configuration so that by Proposition \ref{th:caracEffective}
it is not $L_G$-effective if $a <0$, giving the first part of the formula.

\medskip
For the second part we have for the degree of $f$:
\[
deg(f) = a + \frac{(n-1)(n-2)}{2}
\]
Moreover the number of edges of the complete graph is
$m =  \frac{n(n-1)}{2}$ giving $ deg(f) =a +m- (n-1)$
So that we can apply Proposition \ref{propo:largeDegreeRank}
when $ deg(f) > 2m -2n$, that is when  $a >m-n-1$.
When this condition is satisfied Proposition \ref{propo:largeDegreeRank} 
gives:
$$ \rho(f) = deg(f) -m +n -1 = a$$
It is easy to check that when $a = 0$ the configuration $f$ is $L_G$-effective and 
subtracting 1 to $a$ gives a non $L_G$-effective configuration.
Hence the rank of the configuration $f$ when $a=0$ is 0.
By Corollary  \ref{coro:augmRangDeUn} while adding $1$ to $a$
from $a=0$ to $a = m-n$ the rank increases at most by 1 at each step,
and it has to go from 0 to a,  hence it increases exactly by 1 at each step. This ends the proof.

\end{proof}
\subsection{Algorithm} 
The two Propositions above give a recursive greedy algorithm in order to compute
the rank of a configuration $f$  in $K_n$. It consists in determining first  the parking 
configuration $g$  equivalent to $f$; if $g_n$ is negative then the rank is -1.
If $g_n \geq 0$, use the fact that for a parking configuration  there is necessarily an $i < n$ such that 
 $g_i = 0$, otherwise the configuration is not superstable; then  the rank of the 
configuration $g$ is equal to the rank of $h = g - \eps^{(i)}$ 
increased by 1. This algorithm terminates since one obtains recursively after
at most $deg(g)$ steps a  non $L_G$-effective configuration. Notice  that the rank 
is exactly the number of recursive calls of the algorithm minus 1. A first improvement of this algorithm consists in making it to stop when the configuration attained at some step satisfies the conditions stated in Proposition \ref{propo:caculRangPourLaMax}.

\medskip

The main difficulty  of the algorithm consists in obtaining at each step  the parking configuration 
toppling equivalent to a given one, but this can be simplified supposing that the configurations
are sorted at each step. That is reordering the $f_i$ such 
that  $f_{i-1} \leq f_i$ for all $1 < i < n$, this reordering do not modify the rank.

\bigskip

Let us consider as an example the computation of the rank of the configuration 
 $f = (3, 1, 3, 4, -1)$ of $K_5$. 

 The configuration $ (0, 3, 0, 1, 6)$ is the parking configuration toppling equivalent to $f$.
It is preferable to write the coordinates in increasing order obtaining a new configuration with the same rank:
$(0, 0, 1, 3, 6)$.
 A first step is to subtract $\eps^{(1)}$ giving 
$(-1,0,1, 3, 6)$. Then to reach the toppling equivalent parking configuration, we topple once the sink to obtain $(0,1,2,4,2)$ which is not a 
parking one  due to the fourth vertex and after its toppling we obtain  the expected parking configuration which is $h = (1, 2, 3, 0, 3)$. From Proposition~\ref{propo:caculRangMoinsUn} we have $\rho(f) = \rho(h) +1$. Reordering gives $h' = (0,1, 2, 3, 3)$ and concludes this first step.

\smallskip

At the second step we get the configuration $(-1, 1,2,3,3)$ and the 
parking configuration $(0,1,2,3,2)$ after toppling and reordering. 
The third step gives $(-1,1,2,3,2)$ and the parking configuration 
$(3,0,1,2,1)$. Two other steps are necessary 
to obtain the non $L_G$-effective configuration $(1,2,0,3,-1)$ of rank $-1$
giving $\rho(f) = 4$.

Notice that an application of Proposition \ref{propo:caculRangPourLaMax} would
have given  the result after the first step, since this Proposition \ref{propo:caculRangPourLaMax} gives
$\rho(0,1,2,3,k) = k$ for $k \geq -1$.

A first glance at this algorithm shows that the  complexity of the determination 
of the rank of the configuration  $f$ in $K_n$  is in $O(nD)$, where $D$ is the degree of $f$. 
But this could be lowered  to    $O(n)$ using some  observations on Dyck words.

\subsection{Working on Dyck words} 

We use here the notation of Section \ref{sec:dyck} to which we add some items.
 We  denote by $\eps$ the empty word and 
the {\em height } of  a word $w$  is given    by $\delta(w)$, the value of $ |w|_a - |w|_b$.
Recall that a $w$ is a Dyck word if $wb \in D_n$ where $n = |w|_a+1$. 
The  {\em first return} decomposition of a Dyck word $w$ is given by  $w=aubv$ where $u$ and $v$
are Dyck words. 
The  {\em factorization of a word $w$  of $D_n$ into primitive factors} is given by 
$w= aw_1baw_2b \cdots aw_kbb$ where the $w_i$'s are Dyck words.

Let $f=(f_1, f_2, \ldots, f_{n-1}, f_n)$ be a parking configuration of $K_n$  such that 
$f_1 \leq f_2 \cdots \leq f_{n-1}$, then $f_1=0$. The configuration $ f - \eps^{(1)}$ has a negative value ($-1$) in vertex $x_1$, the 
equivalent configuration $f'= f - \eps^{(1)} - \Delta^{(n)}$ is such that 
$f'_1 = 0$, $f'_n = f_n -(n-1)$ and $f'_j = f_j +1$ for all other values of $j$. 
 Since $f$ is 
a parking configuration, $f'$ satisfies equation (\ref{eqn:prePark})  so that we can apply Proposition \ref{prop:uniqConj}
in order to determine the parking configuration $g$ equivalent to it. This gives:

\begin{propo}
\label{prop:oneStepRank}
For a parking configuration $f$ of $K_n$ such that 
$f_1 \leq f_2 \cdots \leq f_{n-1}$ and $\phi(f)=aubvb$, where $u,v$ are Dyck words, the 
configuration $f - \eps^{(1)}$ is such that:

\begin{equation}
 \phi(f - \eps^{(1)}) = vabub { \hskip 0.5cm }{\rm and }{ \hskip 0.5cm }
\psi(f - \eps^{(1)}) = f_n - |aub|_a
\end{equation}
\end{propo}

\begin{proof}
The existence of the decomposition $\phi(f)=aubvb$ is a consequence of the algorithm computing parking configurations via Dyck words.
Let $f'$ be as above, 
since $f'_j = f_j +1$ for all $j \notin \{i,n\}$ and $f'_1=0$, it is easy  to  check that :
$\phi_1(f') = abubv$, the conjugate of this word which is in  $D_n$  is $ vabub$. 
Applying Proposition \ref{prop:uniqConj} we have
$$ \phi(f') = vabub$$
Remark \ref{rem:calculConjugue} gives
$$ \psi(f') = f'_n + n(q-p) -q$$
where $p = |abub|_a$ and $q=|abub|_b$, giving $q-p = 1$, hence:
$\psi(f') = f'_n + n - |u|_b -2$ and the result follows from $f'_n = f_n -n +1$.

 \end{proof}

\bigskip

  This result above shows that instead of considering the parking  configuration 
   $f$, it is preferable to work with $(w  =\phi(f), s=f_n)$   a   pair consisting of  word $w$ in $D_n$ and an  integer $s$.
    Define 
   two  functions $\theta_1$ and $\theta_2$, where $\theta_1(w)$ is a word in $D_n$  and $\theta_2(w,s)$ is an integer, they are both given
   using the first return to the origin decomposition $aubvb$ of $w$:
   \begin{equation}
   \label{eq:theta}
   \theta_1(aubvb) = vabub   { \hskip 1cm}    \theta_2(aubvb, s) = s- |aub|_a.
   \end{equation}
   Hence, one loop iteration starting from the pair $(w,s)$ leads to the pair 
   \begin{equation}
     \Lambda(w,s) = (\theta_1(w),\theta_2(w,s)).
   \end{equation} 
    So that the rank r  of $f$ is obtained by iterating $\Lambda$ until reaching a negative number for
   the second component of the pair, this may be written:
    $$ r+1 = \min_{k\geq 0} \{\theta_2^k(\phi(f), f_n) < 0\}$$
   where we use  the slight abuse of notation $\theta_2^k(f,s)$ to denote the second component of $\Lambda^k(w,s)$. 
    The algorithm on 
    parking configurations could then be translated in terms of Dyck words as:
   
   \begin{itemize}
        \item $w$ {\tt :=} $\phi(f) $; $s$ {\tt :=} $f_n$; $r$ {\tt := -1}
       \item {\tt  while}  $s\geq 0$ {\tt do}
       
          \hskip 1cm   $s$ {\tt :=} $\theta_2(w,s)$; $w$ {\tt :=} $\theta_1(w) $; $r$ {\tt :=} $r +1$
       
       \item {\tt od }
       \item {\tt return } $r$
    \end{itemize}

   \bigskip
   
   Iterating  this  Proposition a few times gives:
   
   \begin{lemma}
   \label{lem:iterTheta}

   Let $f$ be a parking configuration and 
   $w \in D_n$ be equal to $\phi(f)$.  Let $w = aw_1baw_2b \cdots aw_kbb$ be the factorization  
   of $w$ into its primitive  factors. 
   Then 
   $$ \theta_1^{k} (w) = abw_1abw_2\cdots abw_kb{\hskip 1cm } \theta_2^{k} (w,s) = s-(n-1)$$
   Moreover any  parking configuration $g$ such that $\phi(g) = w_1abw_2 \cdots abw_kab$ and $g_n = s -(n-1) \geq 0$
   satisfies:
   $ \rho(g) = \rho(f) -k$.
   \end{lemma}
   
   In the sequel we will use extensively the following sequence associated to a word $w$ containing $m$ occurrences
   of the letter $a$
   \begin{definition}
   The sequence of heights $\eta_1, \eta_2, \ldots, \eta_m$ of $w$ is such that 
   $$\eta_i = |w^{(i)}|_a - |w^{(i)}|_b$$
   where $w^{(i)}$ is the prefix of $w$ followed by the $i$-th occurrence of the letter $a$.

   \end{definition}
   
   \medskip
   Let $w \in D_n$,    $w=aw_1b \cdots aw_kbb$, be the decomposition of $w$ into primitive factors
   and 
   $\eta_1, \eta_2, \ldots, \eta_{n-1} $ be
    its  sequence of heights.
   Let  $\eta'_1, \eta'_2, \ldots, \eta'_{n-1} $ be that of   $\theta_1^k(w)$, then:
       
     \centerline{$\eta'_i = \eta_i$ if $\eta_i = 0$ and $\eta'_i = \eta_i -1$ if $\eta_i > 0.$}

  \begin{theorem}
  \label{th:calculRangKnDyck}
 Let $f$ be a parking configuration on $K_n$, let $w = \phi(f)$  and $\eta_1, \eta_2, \ldots  \eta_{n-1}$
 be the sequence of heights in $w$.
 Then the rank of $f$ is given by:
 \begin{equation}
 \label{eq:calculRangKnDyck}
1+ \rho(f) =  \sum_{i=1}^{n-1} Max(0, q-\eta_i+\chi(i \leq r))
  \end{equation}
  Where $q$ and $r$ are the quotient and remainder of the division of $f_n+1$ by $n-1$, and $\chi(i \leq r)$ 
  is equal to 1 if $i \leq r$ and to 0 otherwise.
   \end{theorem}

 \begin{proof}
 When $f_n < 0$ all the terms in  the sum in the right hand side of Equation (\ref{eq:calculRangKnDyck}) are equal to $0$
 so that the formula gives $\rho(f) = -1$ as expected.
 
 \medskip
 
 For $f_n \geq 0$, we  proceed by induction on $q$ the quotient of $f_n +1$ by $n-1$ and we
  consider the decomposition of $\phi(f)$ into factors:
   $$\phi(f) = w = aw_1b \cdots aw_kbb$$

  \begin{itemize}
  \item  If $q = 0$ then $0 \leq f_n < n -1 $ hence $\rho(f) $ is determined by the position of $f_n$ in the intervals given by the lengths
  of the prefixes $fw'$  of $w$ satisfying $\delta(w') = 0$ these are $ aw_1baw_2b \cdots aw_ib$  for $i = 1, 2, \ldots k$.
   More precisely let $n_i$ be given for $1 \leq i \leq k$ by 
  $$ n_i = |aw_1baw_2b \cdots aw_ib|_a$$
 and let us   denote $n_0 = 0$.
 Then we have $\theta_2^i(w,f_n) = f_n - n_i $ for $i = 1, 2 , \ldots , k+1$ so that $\rho(f)$ is given by the first
 $i$ for which   $\theta_2^i(w,f_n) $ is negative which may be translated by :
 $$ \rho(f) = i \  \Longleftrightarrow \  n_i \leq f_n < n_{i+1}.$$
 
  On the other hand   the $n_i$ can be characterized by $\eta_{n_i } =0$ giving:
 
  $$1+\rho(f) =  |\{j | j \leq r, \eta_j = 0 \}|$$
  But this is exactly what equation \ref{eq:calculRangKnDyck} gives.
  
  \item If  $q > 0$ let us  consider  $g= \theta_1^{k}(w)$ and notice that $ \theta_2^{k+1} (w,f_n) = f_n - (n-1)$ 
  and let $g$ be such that $\phi(g) = w, g_n = f_n -(n-1)$.
  Let then the quotient of $g_n$ by $n$ is $q' = q-1$ and the remainder is also $r$.
  Let $\eta'_1, \eta'_2, \ldots,  \eta'_{n-1}$ be the sequence of  heights of $v$.
  Applying the inductive hypothesis we have
  $$1+\rho(g) =  \sum_{i=1}^{n-1} Max(0, q-1-\eta'_i+\chi(i \leq r))$$

 By Lemma \ref{lem:iterTheta} , we have $\rho(f) = \rho(g) +k$ and 
  $v =w_1abw_2b \cdots abw_kab$.
  
It is easy to check since the heights $\eta'_i$ in  $g$ are such that $\eta'_i =  \eta_i-1$ if $\eta_i \neq 0$ and 
 $\eta'_i = 0 $ when $\eta_i=0$ we have  
 $Max(0, q-1-\eta'_i+\chi(i \leq r)) = Max(0,q-\eta_i \chi(i \leq r)$ when $\eta_i \neq 0$
 and $Max(0, q-1-\eta'_i+\chi(i \leq r)) = Max(0,q-\eta_i \chi(i \leq r)-1)$ when $\eta_i =0$
 from which we get 
  $$1+\rho(g) =  \sum_{i=1}^{n-1} Max(0,q- \eta_i+\chi(i \leq r))- k$$

  The proof ends by reminding that $\rho(f) = \rho(g) + k$.
\end{itemize}
 \end{proof}
 
 \medskip
 
  An exemple of calculation:
 $$f=(0,0,0, 1,1,1,4,7,7,9, 26)$$
 
 For this configuration  we have: $n = 11, f_n = 26, q = 2, r = 7$.
\vskip 0.2cm
\begin{center}
\begin{tabular}{|c|c|c|c|c|c|c|c|c|c|c|c|c|}
\hline
$i$ &1 & 2& 3 & 4& 5& 6 & 7& 8 & 9 & 10\\

\hline
$f_i$ &0 & 0 & 0 & 1& 1 & 1 & 4 & 7 & 7 & 9\\
$\eta_i$ &0 & 1& 2& 2 & 3& 4& 2 & 0 & 1 & 0\\
$q+\chi(i \leq r)$& 3& 3 & 3 & 3& 3 & 3 & 3 & 2 & 2 & 2\\
\hline 
$q+\chi(i \leq r) -\eta_i$& 3 & 2 & 1 & 1& 0  & -1& 1 & 2  & 1  & 2\\
\hline
\end{tabular}
\end{center}

\medskip

\centerline{Adding the positive values of $q+\chi(i \leq r)$ gives $1+\rho(f)= 13$ so that $\rho(f) = 12.$}
This formula will be rephrased in terms of the representation of a Dyck word by a path in $\Z^2$ in a subsequent 
section and in terms of skew cylinders in the appendix.

\medskip
 
\section{A new parameter for Dyck words}

\subsection{Prerank and coheights}
The algorithm used to determine the rank of a configuration of the complete graph suggests the introduction
of a parameter on Dyck words (or Dyck words) which we call { \em prerank}. In this Section we show how this 
parameter behaves with other known parameters, it may be skipped by the reader only interested by the 
rank on complete graphs.

 We denote $\theta$ the mapping associating to the Dyck word $w$ with first return to the origin decomposition $w = aubv$ the Dyck path $vabu$, notice that $\theta_1(wb) = \theta(w)b$.  And we consider the Dyck  word $(ab)^n$ consisting of the concatenation  of $n$ two-letters words $ab$.
 
\begin{definition} The prerank  $\rho(w)$ of  the Dyck path $w$, of length $2p$, is the  integer $k$ such that $k$ operations $\theta$ are needed to reach the 
word  $ (ab)^p$, or equivalently the smallest $k$ such that: $\theta^k(w) = (ab)^p$. 
\end{definition}
Recall that the {\em area} of a Dyck word  is equal to the sum of the elements  of its sequence of heights.
So that we have $area(aubv) = area(u)+area(v) + |u|_a$.
Notice that $\rho( (ab)^n) = area((ab)^n)=  0$, one can also prove that $\rho(a^nb^n) = area( a^nb^n)= \frac{n(n-1)}{2}$, this suggests that pre-rank and area
could be  equal, but it is not always  the case since  for example $\rho(abaabb) = 2$ while $area(abaabb) = 1$. 

The prerank may be calculated using the notion of {\em coheight}, which may be defined as follows:

Let the heights of the prefixes followed by an occurrence of $a$ of a Dyck word $w$ of length $2p$  be: $\eta_1, \eta_2, \ldots \eta_p$,
 let $m$ be the largest integer such that $\eta_m$ is maximal among the $\eta_i$, then the sequence of  co-heights
 of $w$ 
is given by the formula
\begin{equation}
\overline{\eta}_i = \left\{
\begin{array}{ll}
\eta_m-\eta_i & \mbox{if $i \leq m$}\\
\eta_m-\eta_i -1 &  \mbox{otherwise}
\end{array}
\right. 
\end{equation}

The following lemma allows to determine the prerank
\begin{lemma}
The prerank $\rho(w)$ of the Dyck word $w$ is given by the sum of the  elements of its sequence of coheights.
\end{lemma}

\begin{proof}
We proceed by induction on the value of $\rho(w)$. If $\rho(w) = 0$, then $w = (ab)^p$ and $\eta_i = 0$ for all $1 \leq i \leq n$ hence $\eta_m = 0$
and $m = n$, then  all the coheights are equal to $ 0$ in accordance with the Lemma.
\medskip

Let  $w$ be such that $ \rho(w)\neq 0$ and let $aubv$ be the first return to the origin decomposition of $w$, we have $\rho(w) = 1 + \rho(vabu)$. We proceed to the comparison of the coheights of $w$ and those of $vabu$ considering two cases according to the position $m$ of the largest prefix of $w$ with maximal height. Denote $k = |aub|_a$.
\begin{itemize}
\item If $k< m$ then the sequence of  heights  of $w$ may be written
 $$\eta_1(=0), \eta_2, \ldots, \eta_k (=0), \ldots ,\eta_m, \ldots,  \eta_n$$  
 and that of $vabu$:
$$\eta_{k+1}, \ldots, \eta_m, \ldots,  \eta_n, 0, 0,\eta'_2,  \ldots, \eta'_k$$
 where $\eta'_i = \eta_i -1$ for $i = 2,\ldots k.$

The values of the heights which decrease by one from the first sequence to the second one,  are moved from before $\eta_m$
to after $\eta_m$; hence their corresponding coheights do not vary. 
Those who do not decrease give also equal coheights.  The only 
difference is that $\eta_1=0$ moving from before $\eta_m$ to after it  and remains equal to 0,  so that the corresponding
coheight decreases by one. Hence the sum of the sequence of  coheights of $vabu$ is one less than the sum of the sequence
of those of $f$. Since $\rho(w) = 1 +\rho(vabu)$, the inductive hypothesis proves the  assertion in the Lemma.

 \item If $k  \geq m$.
 Then the sequences of heights may be written:
 $$\eta_1(=0), \eta_2, \ldots, \eta_m, \ldots, \eta_k(=0), \ldots,  \eta_n$$  
for $f$
and 
$$\eta_{k+1} \ldots \eta_n, 0, 0,\eta'_2, \ldots \eta'_m   \ldots \eta'_k$$
for $vabu$,  where $\eta'_i = \eta_i -1$ for $i = 2,\ldots k.$
The maximal height decreases by 1 for $vabu$, so that  the coheights in $w$ and the corresponding ones in $vabu$ are equal
except as above for the first coheight, for which the value decreases by 1. A similar argument as above ends the proof. 
 \end{itemize}

\end{proof}

\subsection{The mapping $\Phi$}
\label{sec:phi}

For a word $w$ of $D_n$   we consider its sequence of heights $\eta_1, \eta_2 , \ldots,  \eta_{n-1}$ and the largest integer $m$ such that
$\eta_m$ is maximal among the $\eta_i$'s and define $\Phi(w)$ as follows:

\begin{definition}
\label{def:PHI}
Consider   the decomposition of $w \in D_n$ as
$w = uv$ such that $|u|_a = m$, where $m$ is defined above,  then $$\Phi(w) = \tilde{u}\tilde{v}$$
\end{definition}
\noindent
We use here the notation: if  $w=x_1x_2 \ldots x_p $  where $x_i$ are letters, then $\tilde{w} =x_p \ldots x_2x_1$.

\medskip

It is not difficult to prove  that the word $\Phi(w)$ is in $D_n$, by showing that if $w'$ is a strict prefix of $w$ then 
$\delta(w') \geq 0$.

\begin{remark}
\label{rem:PhiTild}
For any word $w$ of $D_n$, the word $\Phi(w)$ is the conjugate of $\tilde{w}$ which is a Dyck word followed by an occurence of $b$.
Consequently $\Phi(\Phi(w)) = w$.
\end{remark}

\begin{proof}
With the notation in Definition \ref{def:PHI} we have
$$\tilde{w}=\tilde{v}\tilde{u}.$$
A conjugate of this word  is $\tilde{u}\tilde{v}=\Phi(w)$, and  this word is in $D_n$.
The cyclic lemma insures the unicity of such a word among the conjugates of $\tilde{w}$, 
thus ending the proof.
\end{proof}

We recall the definition of the parameter $\mathrm{dinv}$ of $w$ introduced by M. Haiman \cite{garsiaHaiman} which is obtained 
from the sequence of heights $\eta_1, \eta_2, \ldots , \eta_k$ of a Dyck word $w$ of length $2k$. We use here the same definition 
for $DINV(w)$ when $w \in D_n$, since it makes no difference for the sequence of heights to consider a Dyck word $w$ or the word $wb$ which is in $D_n$.  Also for $w$ in $D_n$ we define $area(w) = area(w') = \sum_{i=1}^{n-1} \eta_i$ and $\rho(w) = \rho(w')$, where $w'$ is given by: $w= w'b$.
\begin{definition}
For a word $w \in D_n$ the parameter $\mathrm{dinv}(w)$  is equal to 
 the number of elements of the set $DINV(w)$ of pairs 
$(i,j)$ such that $i< j$ and $\eta_i = \eta_j$ or $\eta_j= \eta_i -1$.
\end{definition}
The main result of this subsection is:
\begin{propo}
\label{prop:Phi}
The mapping $\Phi$ is an involution, and for any  $w \in D_n$ we have:
$$ \rho(w) = area(\Phi(w)) \hskip 2cm \mathrm{dinv}(w) = \mathrm{dinv}(\Phi(w))$$

\end{propo}
Before proving this Proposition we need to compare  the sequence of heights in a word $w$ and that of the word $\tilde{w}$

\begin{lemma}
\label{lem:hautTilde}
Let $w$ be any word on $\{a,b\}^*$  containing $n$ occurrences of the letter $a$,  $\eta_1, \eta_2, \ldots , \eta_n$
be its sequence of heights, and let $p = |w|_a - |w|_b$.  The sequence $\eta'_1, \eta'_2, \ldots , \eta'_n$ of heights
of $\tilde{w}$ is given by:  $$\eta'_{n-i+1} = p-1 -\eta_i.$$

\end{lemma}

\begin{proof}(of Lemma)
Let $w=uav$ be such that $u$ contains $i-1$ occurrences of the letter $a$, then 
$\eta_i = \delta(u)$ and since $\tilde{w} = \tilde{v}a \tilde{u}$, $|\tilde{v}a |_a = n-i+1$,
$\eta'_{n-i+1} =\delta( \tilde{v}) =\delta(v)$.

Computing $\delta(w)$ we have
$$ \delta(w) = \delta(u) +1 + \delta(v)=p$$
giving 
$$\delta(\tilde{v}) = p -1 - \delta(\tilde{u}).$$

\end{proof}
\begin{proof}(of Proposition \ref{prop:Phi})
Let $\eta_1, \ldots, \eta_m, \ldots, \eta_n$ be the sequence of heights of $w$, and $m$ be the largest integer such that $\eta_m$
is maximal among the $\eta_i$'s. Let $\eta'_i$ be the  sequence of heights in $\Phi(w)$.
Since $\eta'_m =  \eta_m -\eta_1=\eta_m$ and $\eta'_j \leq \eta_m-$ the largest  $j$ such that $\eta'_j$ is maximal 
among the $\eta'_i$ is equal to $m$. 

\medskip

 Denote $w = w'abw"$ where $|w'a|_a = m$ then $\delta(w'a) = \eta_m+1$ and $\delta(w") = -\eta_m$
Then by Lemma \ref{lem:hautTilde}, the sequence of heights $\eta'_1, \eta'_2, \ldots \eta'_m \ldots \eta'_n$
of $\Phi(w)$ is given by:
\begin{equation}
{\eta'}_i = \left\{
\begin{array}{ll}
\eta_m-\eta_{m-i +1}& \mbox{if $i \leq m$}\\
\eta_m + (-\eta_m -(\eta_{n-i +1+m}-\eta_m-1) ) &  \mbox{otherwise}
\end{array}
\right. 
\end{equation}
Since the second value is $\eta_m-1-\eta_{n-i+m+1}$,
this proves that the sequence of heights in $\Phi(w)$ is a rearrangement of the sequence of the coheights of $w$ proving the first part of  the Proposition.

\medskip
For the second part of the Proposition we consider an element $(i,j)$ in 
$DINV(w) $ and denote $a= \eta_i, b=\eta_j$;  in $\Phi(w)$, these values 
$a,b$ become $a',b'$ and are such that $\eta_{i'} = a'$ and  $\eta_{j'} = b'$. We consider three
cases depending on the values of $i$:
\begin{enumerate}
\item If $i,j \leq m$ then
$$a' = \eta_m -a, b'=\eta_m -b, i'= m-i+1, j'= m-j+1$$
so that $j' < i'$ and $(j',i') \in DINV(\Phi(w))$.

\item If $i,j > m$ then
$$a' = \eta_m-1 -a, b'=\eta_m-1 -b, i'= m+n-i+1, j'= m+n-j+1$$
so that $j' < i'$ and $(j',i') \in DINV(\Phi(w))$.

\item If $i \leq m$ and $j >m$ then $i' \leq m, j' >m$
$a'= \eta_m -a$ $b'=\eta_m-1 -b$.
So that $b'=a'-1$ if $a=b$ and $b'= a'$ if $a= b+1$.
Hence $(i',j') \in DINV(\Phi(w))$.
\end{enumerate}
\end{proof}

\subsection{Link with Haglund's function $\zeta$}
In  \cite[page 50]{haglund} J. Haglund introduces a mapping $\zeta$
on Dyck words and shows that this mapping keeps invariant the value of  $\mathrm{dinv}$ . We will give a
relationship between $\zeta$ and $\Phi$.

Before lets us recall how $\zeta(w)$ is build from the sequence of heights $\eta=\eta_1, \eta_2, \ldots \eta_n$ of $w$.
Let $k$ be the maximal values of the $\eta_i$, then the subsequences $\eta^{(0)}, \eta^{(1)}, \ldots, \eta^{(k)},\eta^{(k+1)}$
where $\eta^{(i)}$ contains all the occurrences of $i$ and $i-1$ in $\eta$ in the same order as they appear in $\eta$.
Hence $\eta^{(0)}$ contains only occurrences of $0$, and $\eta^{(k+1)}$ only occurrences of $k$.
Then $\zeta(w)$ is obtained by concatenating the words $\overline{\eta}^{(i)}$ obtained from the ${\eta}^{(i)}$'s 
replacing occurrences of $i$ by a and those of $i-1$ by $b$.

Example for $w = aabaabbabbaabaabbabb$ we have $\eta = (0,1,1,2,1,0,1,1,2,1)$, so that:
$${\eta}^{(0)}= (0,0), {\eta}^{(1)}= (0,1,1,1,0,1,1,1) ,  {\eta}^{(2)}= (1,1,2,1,1,1,2,1),  {\eta}^{(3)}= (2,2) $$
Giving 
$$\overline{\eta}^{(0)}= aa, \overline{\eta}^{(1)}=baaabaaa ,  \overline{\eta}^{(2)}= bbabbbab,  \overline{\eta}^{(3)}= bb$$
and $\zeta(w) = aabaaabaaabbabbbabbb$.

Consider the mapping $R$ on words consisting in writing them from the right to left and replacing $a$ by $b$ and $b$
by $a$ such that $R(w_1w_2 \cdots w_{n-1}w_n) =  \overline{w}_n\overline{w}_{n-1}\cdots \overline{w}_2  \overline{w}_1$ where $\overline{a} =b,
\overline{b}=a$. 

\begin{propo}

The mappings $\Phi$ and $\zeta$ satisfy
$R(\zeta(w) ) = \zeta(\Phi(w))$. 
\end{propo}
In the example above we have $\Phi(w) = aabaabbabbaabaabbbab$ which has for sequence of heights$(0,1,1,2,1,0,1,1,2,0)$
giving $\zeta(\Phi(w))= aaabaaabaabbbabbbabb$ which is equal to $R(\zeta(w))$.

\begin{proof}
Consider the sequence $\eta = \eta_1, \eta_2, \ldots, \eta_n$ of heights of $w$, let  $k$ be the maximal value of the $\eta_i$ and $m$ the largest integer such that $\eta_m = k$. It is convenient for the proof to divide the $\overline{\eta}^{(i)}$'s leading to the definition of $\zeta(w)$
 in two parts
$\overline{\eta}'^{(i)}$ and $\overline{\eta}"^{(i)}$ corresponding  respectively to $(\eta_1, \ldots , \eta_m)$ and to $(\eta_{m+1}, \ldots , \eta_n)$, this gives since $\overline{\eta"}^{(k+1)}$ is the empty word:

$$R(\zeta(w)) = R(\overline{\eta}')^{(k+1)}R(\overline{\eta}")^{(k)}R(\overline{\eta}'^{(k)})\cdots R( \overline{\eta}"^{(0)})
R(\overline{\eta}'^{(0)})$$

Consider now the sequence $\mu= \eta'_1, \eta'_2, \ldots, \eta'_m , \ldots, \eta'_n$ of heights of $\Phi(w) $, we divide also the 
sequences $\mu^{(i)}$ into to two sequences  $\mu'^{(i)}$ and $\mu"^{(i)}$ giving:

$$\zeta(\Phi(w)) = \overline{\mu}'^{(0)}\overline{\mu}"^{(0)} \overline{\mu}'^{(1)}\overline{\mu}"^{(1)}\cdots
\overline{\mu}'^{(k)}\overline{\mu}"^{(k)}\overline{\mu}'^{(k+1)}$$

We have seen in the Proof
of Proposition \ref{prop:Phi}, that

\begin{equation}
{h'}_i = \left\{
\begin{array}{ll}
\eta_m-\eta_{m-i +1}& \mbox{if $i \leq m$}\\
\eta_m - 1-\eta_{m+n-i}  &  \mbox{otherwise}
\end{array}
\right.
\end{equation}

But this exactly means that :
$\overline{\mu}'^{(i)} = R(\overline{\eta}'^{(k+1-i)})$ and
 $\overline{\mu}"^{(i)} = R(\overline{\eta}"^{(k-i)})$.
 
hence proving this Proposition.
\end{proof}
\newpage

\section{Labeling the cells  of a strip}
\label{sec:involution}

In this section we will represent a sorted  parking configuration $f$ of $K_n$ by a Dyck path in a strip
of the plane $\Z^2$ consisting of $n-1$ infinite rows of cells. The  word $\phi(f)$ will be drawn by a path in which 
north steps  correspond to  occurrences of the letter $a$ and east  steps  correspond to occurrences of the letter $b$.
This will allow to determine the rank and the degree of $f$ by counting some cells determined by the value $f_n$.

\medskip
We consider here a strip of $n-1$ rows in the plane  consisting of $n-1$ infinite rows of cells and 
we  label these  cells of this strip  with  the integers in $\Z$,   as follows: we start with the label $0$
given to the cell at the north-west of the origin, then we continue using labels from 1 to $n-2$  along the main diagonal (dashed in our Figure) 
until the line $y=n-1$ is reached; we repeat this labeling giving the label $k+n-1$ to the cell lying
on the west of the cell with label $k$ and continue in the north-east direction.
We also label the cells in the diagonals on the left of the main with negative integers one giving the label $k-(n-1)$ to the
cell on the east of one labelled $k$ so that the labels increase by 1 when allowing a diagonal in the north east direction. 
Notice that the cell at the north of a cell labeled $k$ has label $n+k$. Notice that this labeling process may be considered as an illustration of  the euclidian division by $n-1$, since 
the cell in row $i$ are those whose labels are  equal to $i-1 \mod (n-1)$.

\medskip

\begin{figure}[H]
\begin{center}
\begin{tikzpicture}[rotate=90,xscale=0.65,yscale=-0.65]
\draw[->,red, line width=2] (0,-4)  -> (0,13);
\draw[->,red, line width=2] (-1,0)  -> (12,0);
\draw[red] node at (11.5,0.7) {$y$};
\draw[red] node at (-0.5,12.5) {$x$};
\draw[black!30!white] (0,0) grid (10, 11);
\draw[line width = 1,dashed] (0,-0.2) -- (10, 9.98);
\draw[black!30!white](0, -4) grid (10, 0);
\foreach \x/\y/\color/\currentsink in {0/0/green/0,1/1/white/1,2/2/white/2,3/3/white/3,4/4/white/4,5/5/white/5,6/6/white/6,7/7/green/7,8/8/white/8,9/9/green/9,0/-1/green/10,1/0/green/11,2/1/white/12,3/2/white/13,4/3/white/14,5/4/white/15,6/5/white/16,7/6/green/17,8/7/green/18,9/8/green/19,0/-2/green/20,1/-1/green/21,2/0/green/22,3/1/green/23,4/2/white/24,5/3/white/25,6/4/green/26,7/5/white/27,8/6/white/28,
9/7/white/29,0/1/white/-10,1/2/white/-9, 2/3/white/-8, 3/4/white/-7,4/5/white/-6,5/6/white/-5,6/7/white/-4,7/8/white/-3,8/9/-2,9/10/white/-1,
9/11/white/-11,8/10/white/-12,7/9/white/-13,6/8/white/-14, 5/7/white/-15, 4/6/white/-16,3/5/white/-17,2/4/white/-18, 1/3/white/-19, 0/2/white/-20,
0/-3/white/30,1/-2/white/31,2/-1/white/32,3/0/white/33,4/1/white/34,5/2/white/35,6/3/white/36,7/4/white/37}{
\draw node[black] at (\x+0.5,\y-0.5) {\small $\currentsink$};
}

\end{tikzpicture}
\caption{Labels of the cells in the strip }
 \label{fig:calcRank}
\end{center}
\end{figure}
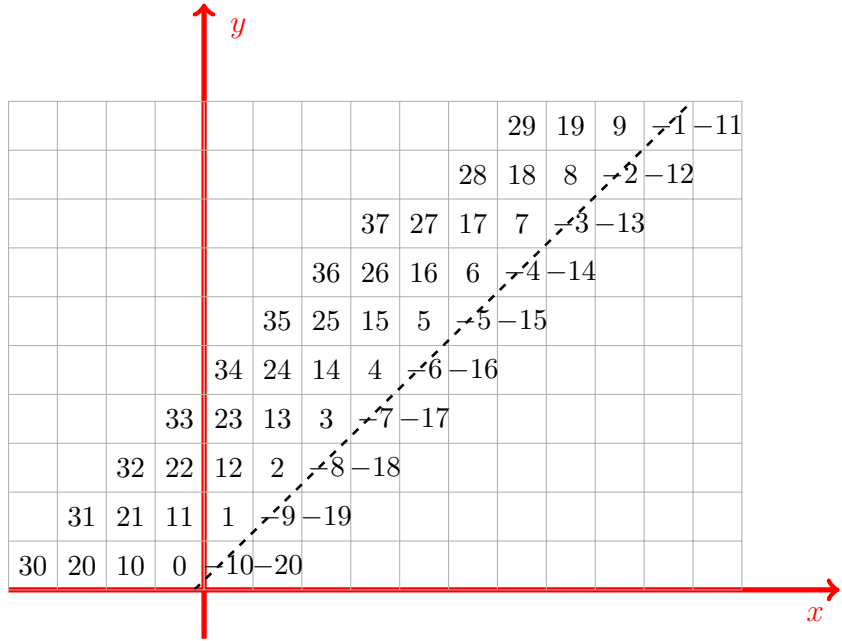
For the representation of the sorted parking configuration $f$ we 
 draw the Dyck path corresponding to $\phi(f)$ starting  at the origin and where a north step corresponds 
  to each occurrence of the letter $a$ and an east  step corresponds to each occurrence of the letter $b$.
Consider each row $i$ ($1 \leq i < n$), composed with the cells lying between the lines $y=i-1$ and $y=i$,
the value of $f_i$ is equal to the number of cells lying between the $y$-axis and the north step of the path in this row.
In addition, $\eta_i$ is the number of cells in row $i$ between the North step
of the Dyck path crossing this row and the main diagonal. Consider the euclidian quotient $q$ of 
$f_n+1$ by $n-1$ such that $f_n+1 = q(n-1) +r$ then the value
 $\max(0,q+\chi(i\leq r)-\eta_i)$ given in Theorem \ref{th:calculRangKnDyck} is the number of  these cells situated on the 
 left of the North step in row $i$ which labels are  less or equal to $f_n$.
 
 This fact is illustrated in Figure \ref{fig:calcRank}. 

\begin{figure}[H]
\begin{center}
\begin{tikzpicture}[rotate=90,xscale=0.65,yscale=-0.65]
\draw[->,red, line width=2] (0,-4)  -> (0,12);
\draw[->,red, line width=2] (-1,0)  -> (12,0);
\draw[red] node at (11.5,0.7) {$y$};
\draw[red] node at (-0.5,11.5) {$x$};
\draw[black!30!white] (0,0) grid (10, 11);
\draw[line width = 1,dashed] (0,0) -- (10, 10);
\draw[line width = 1,dashed,blue] (-1, 0)--(11, 0);
\foreach \x/\fx in {0/0,1/0,2/0,3/1,4/1,5/1,6/4,7/7,8/7,9/9}{
\draw node at (\x+0.5,12) {$\fx$};
}
\draw node at (11, 12){\tikz{\draw[line width = 2] (0,0) circle (0.3);\draw node at (0,0) {$26$};}};
\foreach \x/\fx/\fnx in {0/0/0,1/0/0,2/0/1,3/1/1,4/1/1,5/1/4,6/4/7,7/7/7,8/7/9,9/9/10}{
\draw[blue,line width=2] (\x,\fx) -- (\x+1,\fx);
\draw[blue,line width=2] (\x+1,\fx) -- (\x+1,\fnx);
}
\draw[black!30!white](0, -4) grid (10, 0);
\foreach \x/\y/\color/\currentsink in {0/0/white/0,1/1/white/1,2/2/white/2,3/3/white/3,4/4/white/4,5/5/white/5,6/6/white/6,7/7/green/7,8/8/white/8,9/9/green/9,0/-1/green/10,1/0/green/11,2/1/white/12,3/2/white/13,4/3/white/14,5/4/white/15,6/5/white/16,7/6/green/17,8/7/green/18,9/8/green/19,0/-2/green/20,1/-1/green/21,2/0/green/22,3/1/green/23,4/2/white/24,5/3/white/25,6/4/green/26}{
\draw node[black] at (\x+0.5,\y-0.5) {\small $\currentsink$};
}
\foreach \x/\y in {0/0,7/7,9/9,0/-1,1/0,7/6,8/7,9/8,0/-2,1/-1,2/0,3/1,6/4}{
\draw[opacity=0.6,fill=green,draw=green] (\x+0.5,\y-0.5) circle (0.4);
}
\foreach \x/\rx in {1/3,2/2,3/1,4/1,5/0,6/0,7/1,8/2,9/1,10/2}{
\draw[green!50!brown] node at (\x-0.5,-5) {$\rx$}; 
}
\draw[green!50!brown] node at (0-1,-4.5) {$=13$};
\draw[line width=2] (7, 10) -- (7, -5.5);
\draw[->,line width=2] (0,-5.7) -> (7,-5.7);
\draw node at (3.5,-7.2) {$r=7$};
\draw[->,line width=2] (10.3,9) -> (10.3,7);
\draw node at (11.2,8) {$q=2$};
\draw[blue,line width=2] (10, 11) -- (10, 10);

\end{tikzpicture}
\caption{Computation of the rank on a Dyck path}
 \label{fig:calcRank2}
\end{center}
\end{figure}
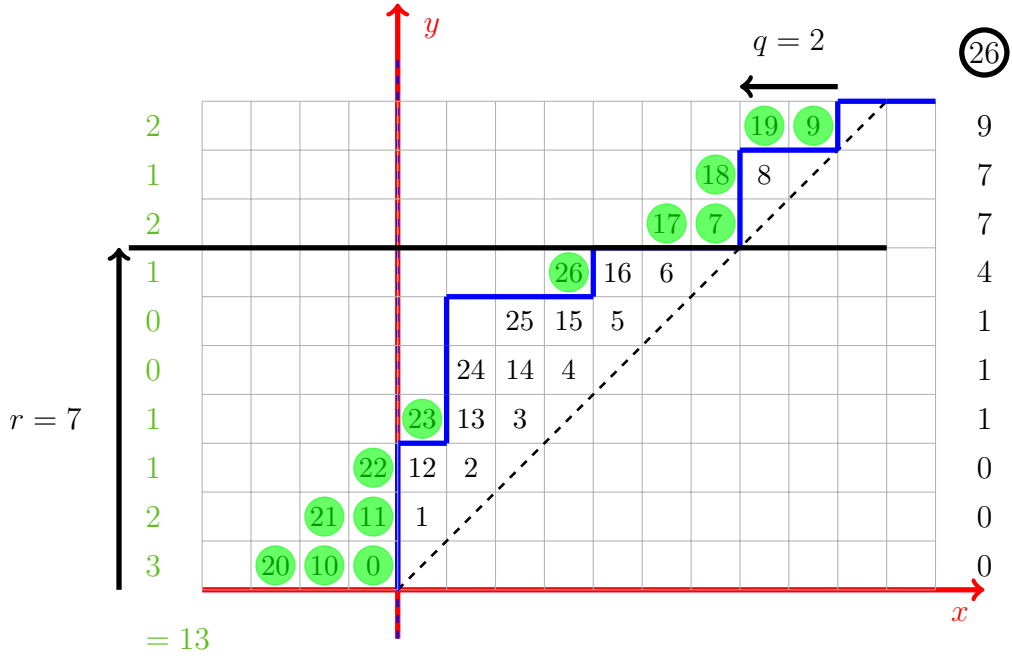

Our configuration $f$ is represented by the Dyck word
$aaabaaabbbabbbaabbabb$ and the value $f_n=26$, where $n=11$ for our
example.  
On the right of the Figure one can read the configuration $f$ from bottom to top, each value $f_i$
corresponds to the number of cells between the $y$-axis and the north step in the row $i$.
On the left are displayed the values of $q+ \chi(i \leq r) - \eta_i$ determined
by the cells at the left of the path lying which coordinates $(i,j)$ satisfy $i \geq j-1-q$ 
for $j \leq r$ and $ i \geq j-q$ for $j>r$. The number of these cells colored in green is 
equal to $\rho(f)+1$.

\begin{remark}
\label{rem:rankLeft}
 Theorem \ref{th:calculRangKnDyck} may be interpreted by stating that the rank $\rho(f)$
 is one less than the number of cells lying on the left of the Dyck path associated to $f$ and having a label not greater than $f_n$.
\end{remark}

\
The Dyck path separates the strip of height $n-1$  into two regions, one on the left of  and the other
on its right. Consider the following sets of integers associated to a Dyck path $w$.

\begin{definition}
For a   path $w$ in $D_n$ we define $\mathrm{lastright}(w)$  as the  highest label of a cell lying in the right region.
For any integer $s$, $\mathrm{left}(w,s)$ is the number of cells lying in the left region whose labels are not greater than $s$, and 
$\mathrm{right}(w,s)$ is the number of cells in the right region whose labels are strictly greater than $s$. 
\end{definition}

\begin{lemma}
\label{lem:pivot}
Let $w \in D_n$ be such that $w=uv$ where $u$ is the longest prefix of $w$ with maximal value of $\delta(u)$. Then the value of $\mathrm{last right}(w)$ is equal  to $n(q-p) -q -1$ where $p=|u|_a$ and $q=|u|_b$
\end{lemma}
\begin{proof}
Notice that the   labels of the cells increase by $n$ while going from one cell to that at the top of it and by $n-1$ when going from one cell
to that on that at the left of it. Hence the cell situated at the right of the unitary path joining  $(x,y)$ to $(x, y+1)$ has label equal to 
$ny -(x+1)(n-1)$.
In each row the cell with highest label lying in the right region is stated immediately at the right of the north step of the path in the 
row. This label is higher in row $i$ than in row $j$ if the heights $\eta_i$ and $\eta_j$ satisfy $\eta_i > \eta_j$ or if
 $\eta_i = \eta_j$ and $i>j$, this shows that $\mathrm{last right}(w)$ is the label of the cell in the row corresponding to the longest  prefix of $w$
 with maximal height. 
\end{proof}
\begin{propo}
\label{prop:xy}
Let $f$ be  any sorted parking configuration  of $K_n$, and let $w = \phi(f)$ 
then:
\begin{equation}
\label{eq:lw}
\mathrm{left}(w, f_n) = \rho(u)+1
\end{equation}
and
\begin{equation}
\label{eq:rw}
\mathrm{right}(w,f_n) = {n-1 \choose 2} + \rho(f)- deg(f).
\end{equation}
\end{propo}

\begin{proof}

  The first equaltion follows immediately from Remark \ref{rem:rankLeft}.
  We discuss two cases for the computation of parameter $\mathrm{right}$.

  If $f_n\geq 0$ the number of elements of the set of the right cells  with non negative labels
  is  equal to the area of the Dyck word $w$ which is equal to
  $$ \sum_{i=1}^{n-1} \eta_i =  {n-1\choose 2} - \sum_{i=1}^{n-1} f_i = {n-1\choose 2} - deg(f) +f_n$$
  Among these cells $(f_n +1 - \mathrm{left}(w,f_n))$ have labels not greater than $f_n$. This gives
  $$ \mathrm{right}(w,f_n) = {n-1\choose 2} - deg(f) +f_n -(f_n+1-\mathrm{left}(w,f_n)) = {n-1\choose 2} - deg(f) + \rho(f)$$

   If $f_n<0$, the set cells in the right region with label greater than $f_n$  are those the cells between the Dyck path  and the main diagonal,  and in addition all the
   cells with  negative labels $-1,-2,\ldots, f_n+1$ hence:
 $$\mathrm{right}(w,f_n) =\left({n-1 \choose 2} - (deg(f)-f_n) \right) -( f_n +1 ) $$
 as expected since in this case $\rho(f)=-1$.

\end{proof}
The calculation of $\rho(f)$ and $deg(f)$ is illustrated below with the same configuration $f$ but with $f=13$.
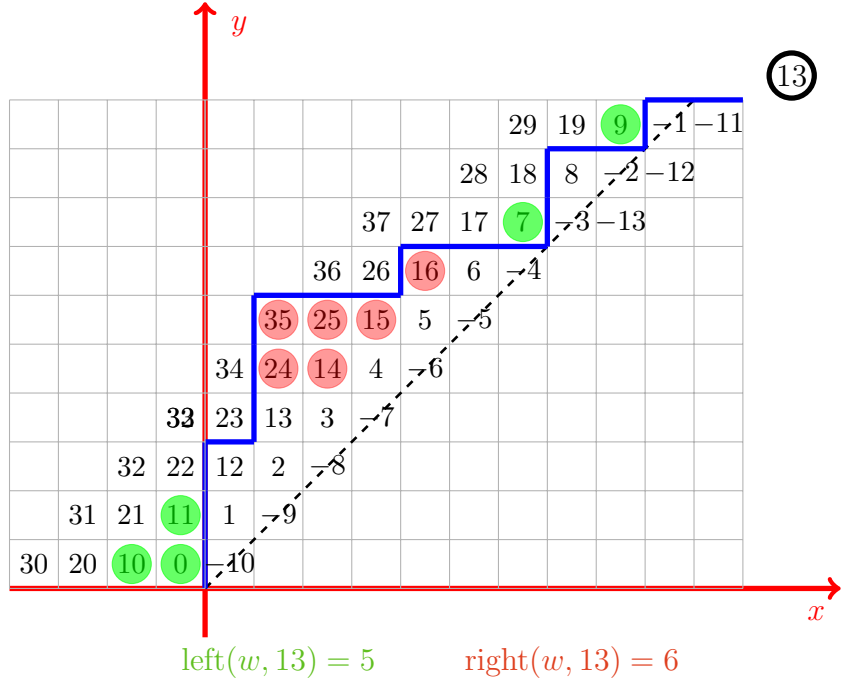
\begin{figure}[H]
\begin{center}
\begin{tikzpicture}[rotate=90,xscale=0.65,yscale=-0.65]
\draw[->,red, line width=2] (0,-4)  -> (0,13);
\draw[->,red, line width=2] (-1,0)  -> (12,0);
\draw[red] node at (11.5,0.7) {$y$};
\draw[red] node at (-0.5,12.5) {$x$};
\draw[black!30!white] (0,0) grid (10, 11);
\draw[line width = 1,dashed] (0,0) -- (10, 10);
\draw node at (10.5, 12){\tikz{\draw[line width = 2] (0,0) circle (0.3);\draw node at (0,0) {$13$};}};,
\foreach \x/\fx/\fnx in {0/0/0,1/0/0,2/0/1,3/1/1,4/1/1,5/1/4,6/4/7,7/7/7,8/7/9,9/9/10}{
\draw[blue,line width=2] (\x,\fx) -- (\x+1,\fx);
\draw[blue,line width=2] (\x+1,\fx) -- (\x+1,\fnx);
}
\draw[black!30!white](0, -4) grid (10, 0);
\foreach \x/\y/\color/\currentsink in {0/0/green/0,1/1/white/1,2/2/white/2,3/3/white/3,4/4/white/4,5/5/white/5,6/6/white/6,7/7/green/7,8/8/white/8,9/9/green/9,0/-1/green/10,1/0/green/11,2/1/white/12,3/2/white/13,4/3/white/14,5/4/white/15,6/5/white/16,7/6/green/17,8/7/green/18,9/8/green/19,0/-2/green/20,1/-1/green/21,2/0/green/22,3/1/green/23,4/2/white/24,5/3/white/25,6/4/green/26,7/5/white/27,8/6/white/28,
9/7/white/29,0/1/white/-10,1/2/white/-9, 2/3/white/-8, 3/4/white/-7,4/5/white/-6,5/6/white/-5,6/7/white/-4,7/8/white/-3,8/9/-2,9/10/white/-1,
9/11/white/-11,8/10/white/-12,7/9/white/-13,0/-3/white/30,1/-2/white/31,2/-1/white/32,3/0/white/32,3/0/white/33,4/1/white/34,5/2/white/35,6/3/white/36,7/4/white/37}{
\draw node[black] at (\x+0.5,\y-0.5) {\small $\currentsink$};
}
\foreach \x/\y in {0/0,7/7,9/9,0/-1,1/0}{
\draw[opacity=0.6,fill=green,draw=green] (\x+0.5,\y-0.5) circle (0.4);
}
\foreach \x/\y in {4/2,4/3,5/2,5/3,5/4,6/5}{
\draw[opacity=0.4,fill=red,draw=red] (\x+0.5,\y-0.5) circle (0.4);
}

\draw[green!50!brown] node at (-1.5,1.5) {$\mathrm{left}(w,13) =5$};
\draw[red!50!brown] node at (-1.5,7.5) {$\mathrm{right}(w,13) =6$};

\draw[blue,line width=2] (10, 11) -- (10, 10);

\end{tikzpicture}
\caption{Labels of the cells in the strip  and the computation of $\mathrm{left}$ and $\mathrm{right}$ for $f_n= 13$. }
 \label{fig:calcRank3}
\end{center}
\end{figure}

Baker and Norine's theorem applied to the configuration $f$ involves the computation of the rank for the configurations $f$ and $\kappa-f$.
The map $f\longrightarrow \kappa-f=\Psi(f)$ is an involution.

The sorted parking configuration toppling and permuting equivalent to any configuration $f\in K_n$ is described by $(\phi(f),\psi(f))$ in Definition~\ref{def:phi}.
   
See the appendix for an alternative approach. For the readers who skipped the preceding section we recall that for a word $w$ in $D_n$
we denote $\Phi(w)$ the conjugate of $ \tilde{w}$ which is an element of $D_n$. Hence $\Phi(w) = \tilde{u}\tilde{v}$, such that $w=uv$ and  $u$ is the longest
prefix of $w$ with maximal value by $\delta$.

\begin{propo}
\label{prop:parkingEquivMoinsU}
Let $f$ be any configuration on $K_n$ and $w= \phi(f)$, then
$\phi(\kappa-f) = \Phi(w)$.
Moreover $ \psi(\kappa-f) =  \mathrm{lastright}(w) -1-\psi(f)$

\end{propo}

\begin{proof} 

Let $f'$ be the parking configuration toppling  equivalent to $f$, so that it satisfies  $w = \phi(f')$, let us denote
 $g$ the configuration $\kappa -f'$ then we have $g \simil \kappa -f$ and using $\kappa = (n-3, n-3, \ldots, n-3, n-3) \simil (n-2, n-2, \ldots, n-2, -2)$:
$$ g=(n-2-f'_1, n-2-f'_2, \ldots n-2-f'_{n-1}, -2-f'_n).$$

\medskip

Notice that, since if $i<n$ we have $0 \leq f'_i < n-1$, the $g_i$ satisfy $0 \leq g_i < n$ for $1 \leq i <n$. Moreover $f'_i  \leq f'_j$
if and only if $g_i \geq g_j$.
This last fact implies  that  the configurations $f"$ and $g'$ obtained by reordering the first $n-1$ values of 
$f'$ and $g$ respectively satisfy $g'_i = n-2-f"_{n-i}$.
A consequence is that  the words $w = \phi(f)=\phi_1 (f")$ and $v= \phi_1(g')$
are such that:
$$ w= ubb \hskip 0.2cm \mbox{\rm  and } \hskip 0.2cm v = \tilde{u}bb$$
Hence  $v$ is a conjugate of $\tilde{w} = bb\tilde{u}$.

\medskip

By Remark \ref{rem:PhiTild} we have that $\Phi(w)$ is the conjugate of $v$ which is in $D_n$,
and it   is also equal to $\phi(g')$ by Proposition \ref{prop:uniqConj}, since $v = \phi_1(g')$.
Now since $g'$ is obtained by permuting the first $n-1$ values of $g$ and $g \simil \kappa -f$ we have:
$$ \Phi(w) = \phi(g') =  \phi(\kappa -f).$$

\bigskip
We now compute $\psi(\kappa -f)$. We consider two configurations satisfying condition (\ref{eqn:prePark}) and 
use Remark \ref{rem:calculConjugue}. These two configurations are $g'$ and  the parking configuration $g"$  toppling equivalent to $g'$, 
indeed in each one the first $n-1$  values
are in weakly increasing order and they satisfy condition (\ref{eqn:prePark}). Moreover  their degrees are equal since they are 
toppling equivalent. Notice that $g'_n = g_n =  -2-\psi(f)$ and $g"_n = \psi(\kappa -f)$. In order to apply Remark \ref{rem:calculConjugue} we have
to compare $\phi_1(g')=v= \tilde{u}bb$ and $\phi_1(g") = \Phi(w)$ which are two conjugates word, the second one being
an element of $D_n$. Let $v_1$ be the smallest  prefix of $v$ which has minimal value by $\delta$ then:
$$v = v_1v_2 \hskip 0.5cm  \mbox{\rm and }  \hskip 0.5cm \Phi(w) = v_2v_1$$
Now denote $p= |v_1|_a$ and $q= |v_1|_b$, and by Remark \ref{rem:calculConjugue} we get:
$$ \psi(\kappa -f) = g"_n = g'_n + n(q-p) -q = -\psi(f)+n(q-p) -(q+2)$$

Using Lemma \ref{lem:pivot} we have

$$ \mathrm{lastright}(w) =n(q-p) -(q+1) $$
which ends the proof.

\end{proof}

\begin{coro}
The involution associating $(w,s)$ to $(w',s')$, where $w'= \Phi(w)$ and $s' =\mathrm{lastright}(w) - s-1$ is such that
$$ \mathrm{left}(w,s) =  \mathrm{right}(w',s') \hskip 0.5cm \mbox{\rm and} \hskip 0.5cm \mathrm{left}(w',s') =  \mathrm{right}(w,s)$$
\end{coro}

\begin{proof}
We give here a short proof using Riemann-Roch theorem for graphs (i.e.  Theorem \ref{th:RR}) , a more direct proof is proposed in the appendix.
Let $f$ be a  configuration such that $\phi(f) = w$ and $\psi(f) = s$, this configuration may be supposed to be
a parking configuration such that its first $n-1$ values are sorted, hence we have $f_n=s$ and $w'= \phi(\kappa-f)$, $s'= \psi(\kappa-f)$.
Theorem \ref{th:RR} gives:
\begin{equation}
\rho(f) = \rho(\kappa -f) + deg(f) - {n \choose 2} +n
\end{equation}
Giving:
\begin{equation}
\label{eq:propSym}
\rho(\kappa -f) = \rho(f) - deg(f) +{{n-1} \choose 2} -1
\end{equation}
By the first part of Proposition \ref{prop:parkingEquivMoinsU} we have that $\rho(\kappa - f) =\mathrm{left}(w',s') - 1$ and that the right 
hand side of equation (\ref{eq:propSym}) is equal to $\mathrm{right}(w,s) -1$ this gives $\mathrm{left}(w',s')  = \mathrm{right}(w,s)$.
Using the fact that the mapping exchanging  $(w,s) $ and $(w',s')$ is an involution ends the proof.
\end{proof}

\section{Enumerative study of the  $(degree,rank)$ bistatistic on sorted parking configurations}
In this section, we are interested in the distribution of  the degrees and ranks on sorted configurations on $K_n$. 
For this purpose  we consider the following generating function
$$ K_n(d,r) = \sum_{f} d^{deg(u)}r^{\rho(f)}.$$

In this sum, $f$ runs over all parking configurations of $K_n$  such that $f_1 \leq f_2 \leq \ldots \leq f_{n-1}$.
Notice that this generating function is a formal sum   which  is a Laurent series,  since  $\rho(f)$ may be equal to $-1$ and $deg(u)$ may  be any integer in $\mathbb{Z}$.

Lorenzini \cite{lorenzini} consider a similar generating function summed over parking configurations, not necessarly sorted, and
call it the two variable zeta function of the underlying graph.  

In preceeding section, we remark two linear combinations of the rank and degree statistics are more natural with respect to the involution $f\longrightarrow \kappa-f$ appearing in Riemann-Roch theorem for graphs.
This leads to a change of the variables $r,d$ of $K_n(r,d)$ into variables $x,y$ of a new generating function 

$$ L_n(x,y)=\sum_{w \in D_n}\sum_ {s= -\infty}^{+\infty}  x^{\mathrm{left}(w,s)}y^{\mathrm{right}(w,s)}$$
 which is a more tractable power series in $x$ and $y$. 
In addition, the involution $\Psi$ in preceeding section shows that this new series satisfies  the symmetric relation  $$L_n(x,y)=L_n(y,x).$$

 Our main result is that via this ordering of summations $L_n(x,y)$
 may be described by a sum of geometric sums of two alternating ratio, which share as a common factor the 
 rational  fraction $H(x,y)=\frac{1-xy}{(1-x)(1-y)}$. 
 
We conclude this enumerative study by the summation on all these binomial words  that is related, up to a rescaling due to change of variables, to $\sum_{n\geq 1} K_n(d,r)z^{n-1}$ and appears to be a rational function in $x$,$y$,$z$ and two copies of a Carlitz $q$-analogue of Catalan numbers, each counting area below Dyck paths.

\subsection{Change of variable from $K_n(d,r)$ to $L_n(x,y)$}

The relation between the bistatistics $(deg,\rho)$ and
$(\mathrm{lw},\mathrm{rw})$ are reversible.  Hence the generating
function $K_n(d,r)$ is equivalent up to a change of variables to the
generating function $L_n(x,y)$ defined above.
\begin{coro}(of Proposition~\ref{prop:xy})
For any $n\geq 2$, we have 
\begin{equation}
K_n(d,r) = \frac{d^{{n-1 \choose 2}-1}}{r}L_n(\frac{1}{d},rd).
\end{equation}
\end{coro}

\subsection{Sums on all the $D_n$.}

Our first finite description of $L_n(x,y)$ requires some additional definition on words.
For any word $f$ on the alphabet $\{a,b\}$ we define its weight $W(f)$ which is a monomial in $x$ and $y$ as follows.
Consider the heights $\eta_1, \eta_2, \ldots \eta_p$ of the prefixes of $f$ followed by the occurrence of a letter $a$ and let 
$\alpha = \sum_{\eta_i\geq 0} (\eta_i+1)$, $\beta = \sum_{\eta_i <0} (-\eta_i -1)$ then $W(f) = x^{\alpha}y^{\beta}$.

\begin{propo}\label{prop:toxy}
Let $A_n$ be the set of words having $n$ occurrences of the letter $b$ and $n-1$ occurrences of the letter $a$.
For any $n\geq 2$, we have 
$$ L_n(x,y) =  H(x,y) \left( \sum_{bfb\in A_n} W(bfb)  - \sum_{afa\in A_n} W(afa) \right)$$
where
$$ H(x,y)=\frac{(1-xy)}{(1-x)(1-y)}=1+\sum_{i\geq 1} (x^i+y^i).$$
\end{propo}

\subsubsection{Mixing geometric sums delimited by $aa$ and $bb$ factors of the word $wb$}

The first step in the proof of this proposition sums the bistatistic
$(\mathrm{lw},\mathrm{rw})$ on all cells of a fixed cut skew cylinder
$\mathrm{Cyl}[w]$, here defined by a Dyck word $w$ of size $n-1$.
A cell with label $k<0$,  is in the  right region.
If the cell has  label $k$ satisfying  $k>1+n(n-3)$, then  it is at the left of the triangle $T(n)$ of corners $(0,0)$, $(n-1,0)$ and $(n-1,n-1)$ which contains the $w$ cut path hence it is  in the left region.
We thus know that there is an odd number $2m+1$ ($m\geq 0$) of indices $k_1\leq k_2 \leq \ldots k_{2m+1}$  such that 
the cells with labels $k_i$ and $k_i+1$ are in different regions. 
In other words the finite sequence $k_i$ for $ {i=1\ldots 2m+1}$ collects the labels of the cells just before   a crossing of the path $w$.
On the example in Figure~\ref{fig:calcRank3}, $m=5$ and $k_1,k_2,\ldots k_{2m+1} = -1, 0,6,7,8,11,16, 23, 25,34,35.$  
We will call the  integer $m$, the \emph{crossing multiplicity} and the sequence $k_1,\ldots k_{2m+1}$ the \emph{crossing indices} of the Dyck word $w$. 

\begin{lemma}
\label{lem:fromuleCroisements}
Let $w$ be a Dyck word, we consider the sequence ${k\in\mathbb{Z}}$ in the cut skew cylinder $\mathrm{Cyl}[w]$, $m$ its crossing multiplicity and $k_1,\ldots k_{2m+1}$ its crossing indices. For any $k \in \mathbb{Z}$ consider the cell with label $k$ and  denote
$M(k)$ the  monomial $x^{rw(k)}y^{lw(k)}$
We have 
$$ \sum_{k\in\mathbb{Z}} M(k)  = H(x,y)\left( \sum_{i=0}^m M(k_{2i+1})-\sum_{i=1}^m M(k_{2i})\right).$$
\end{lemma}   

\begin{proof}
Observe first the two   following local facts:
\begin{itemize}
\item 
If $k$ is the label of a cell in the left region, then $lw(k) = lw(k-1)+1$ and $rw(k) = rw(k-1)$, so that:
$$  x^{rw(k)}y^{lw(k)} = y\times x^{rw(k-1)}y^{lw(k-1)}$$
\item 
If $k$ is the label of a cell in the right region, then $lw(k) = lw(k-1)$ and $rw(k) = rw(k-1) - 1$ so that 
$$x^{rw(k)}y^{lw(k)} = \frac{1}{x}\times x^{rw(k-1)}y^{lw(k-1)}.$$
\end{itemize}

\bigskip

 Let $a < b$,   assume that for any $k$ such that $a+1 \leq k \leq b$, the cell  with label $k$  lies in the  left region, then
 a repeated use of the first observation shows that:
$$\begin{array}{lcl}
 \displaystyle \sum_{k=a+1}^{b} M(k) & = & \displaystyle \sum_{k=1}^{b-a} y^{k}M(a)\\
\ &  = & \displaystyle {\frac{y-y^{b-a+1}}{1-y}}M(a)\\
\ &  = & \displaystyle \frac{y}{1-y}\left( M(a)-M(b)\right) \\
\end{array}$$
where in the first equality we use  that each  cell with label $k$ lies the  left region for $k=a+1, \ldots,  b$ and  the third equality uses  $M(b) = y^{b-a} M(a)$.

In the particular case where $b=+\infty$ the term $M(b)$ vanishes, so that:
$$  \sum_{k=a+1}^{\infty} M(k) = \frac{yM(a)}{1-y}$$

On the other hand , assume that for $k=a+1,\ldots, b$ each cell with label $k$ is in the  right region, similarly we have:
$$\begin{array}{lcl}
 \displaystyle \sum_{k=a+1}^{b} M(k) & = & \displaystyle \sum_{k=1}^{b-a} \left(\frac{1}{x}\right)^{k}M(a) \   =  \displaystyle \sum_{k=0}^{b-a-1} x^{k} M(b) \\
\ &  = & \displaystyle  \frac{1-x^{b-a}}{1-x}M(b)\\
\ &  = & \displaystyle \frac{1}{1-x}\left( M(b) - M(a)\right)\\
\end{array}$$

In the particular case where $a=-\infty$, the terms $M(a)$ vanishes so that:
$$  \sum_{k=-\infty}^{b} M(k) = \frac{M(b)}{1-x}$$

We split the sequence of cells with labels ${k\in\mathbb{Z}}$ into maximal factors delimited by its sequence of crossing indices where all cells lie alternatively in the  right region and  the left region, beginning by the right one and ending with the left one.

$$ \sum_{k\in\mathbb{Z}} M(k) =  \displaystyle \sum_{k=-\infty}^{k_1} M(k) + \sum_{i=1}^{2m} \left(\sum_{k=k_i+1}^{k_{i+1}} M(k) \right) + \sum_{k=k_{2m+1}+1}^{+\infty} M(k)$$

For $ k \leq k_1$ the cells of label $k$ are in the right region, and for $k > k_{2m+1}$ they are in the left region  so that:
$$ \sum_{k \leq k_1} M(k) =  {\frac{M(k_1)}{1-x}} \hskip 0.3cm  \mbox{\rm and } \hskip 0.1cm  
\sum_{k>k_{2m+1}} M(k) = {\frac{yM(k_{2m+1})}{1-y}}$$

For $k_{2i-1} <  k \leq k_{2i}$ the cells of label $k$ lie in the left region so that:

$$ \sum_{k=k_{2i-1}+1}^{k_{2i}} M(k)  = \frac{y}{1-y}\left( M(k_{2i-1})-M(k_{2i})\right) $$

For $k_{2i} <  k \leq k_{2i+1}$ the cells of label $k$ lie in the right region so that:

$$ \sum_{k=k_{2i}+1}^{k_{2i}+1} M(k)  = \frac{1}{1-x}\left( M(k_{2i+1})-M(k_{2i})\right) $$

Observe that in  the sum $ \sum_{k\in\mathbb{Z}} M(k)$  the monomials $M(k_{2i})$ appear with a minus sign
and the monomials  $M(k_{2i+1})$ with a plus sign, moreover each one appears twice with factors
$\frac{y}{1-y}$ and $\frac{1}{1-x}$ adding these two factors gives $H(x,y)$  hence:
$$  \sum_{k\in\mathbb{Z}} M(k)= \left( \sum_{i=0}^{m} H(x,y)M(k_{2i+1})\right) + \left(\sum_{i=1}^{m} (-H(x,y))M(k_{2i})\right) 
$$

\end{proof}

\subsubsection{Evaluation of the monomials for the crossing indices of the Dyck  path} 

In this section we translate the evaluation of $M(k_i)$ by a formula involving the Dyck paths.

We first observe that a crossing point in the Dyck path described by the word 
$w$ corresponds to a factor $aa$ or $bb$ in $w$. Indeed since the Dyck path consists of North steps and East steps,
the curve defined by a trajectory consisting in visiting the cells in order of their labels can only cross the path
when there are two consecutive steps in the same direction. Hence the crossing cells $k_{2i}$ correspond to  factors $aa$ in $w$
while the crossing cells $k_{2i+1}$  correspond to factors $bb$. More precisely  a cell with label $k_{2i}$ is immediately 
on  the left  of 
the a North step of the Dyck path followed by another North step, while  a cell with label $k_{2i+1}$ is just below  of 
an East  step of the Dyck path followed by another East step. In order to take also into account the case
$k_1 = -1$ it is convenient to add an occurrence of $b$ at the end of $w$, that is to consider the word $wb$ instead of $w$.

In the example considered above the correspondence   between  factors and crossing cells is given by the
following table:

\begin{center}
\begin{tabular}{ccccccccccccccccccccc}
a&a&a&b&a&a&a&b&b&b&a&b&b&b&a&a&b&b&a&b&b\\
0&11&&&23&24&&35&25&&&16&6&&7&&8&&&-1
\end{tabular}
\end{center}

In the sequel we will consider that the label $k_i$ of a crossing cell  is also the label
of the corresponding occurrence of a  letter in the Dyck word. We will also define the conjugate $\tau(wb,k)$ of Dyck word $w$
followed by a letter $b$, where $k$ is the label of the occurrence
of $b$ appearing in $w=w_1bw_2$  as the the word $bw_2w_1$.

\medskip

Let us examine more precisely what we  do in order to compute  $rw(k)$ and $lw(k)$ for a cell with label $k$ corresponding
to an occurrence of the factor $bb$.  This computation may be decomposed as the  sum of 
$n-1$ values $rw_j(k)$ and $lw_j(w)$ each one corresponding to one row $j$ of the strip.
To determine  $rw_j(k)$  we count  the number of cells with coordinates $(i,j)$ which lie in the right region
and have a label greater than $k$, similarly $lw_j(k)$  is the number of cells with coordinates $(i,j)$ which lie in the left region
and have a label less than or equal to  $k$. Since labels of cells in the same row decrease from left to right,
for each $j$ at least oneof  two values $rw_j(k)$ and $lw_j(k)$ is equal to 0.

In the following Lemma we use the definition of $W(f)$ for a word $f$ given before Proposition \ref{prop:toxy}.

\begin{lemma}
\label{lem:conjuguerPourFacteurxx}
Let $w$ be a Dyck word and $k$
be a  crossing index of $w$ in the cut skew cylinder $\mathrm{Cyl}[w]$.
Then  we have 
$$M(k_i) = W(\tau(w,k_i))$$
Where $\tau(w,k_i)$ is the conjugate of the word $wb$ given as above  by the position of the letter with label
$k_i$, 
\end{lemma}

\begin{proof}
Let $(i_1,j_1)$ be the coordinates of the cell with label $k$ which we suppose to correspond to a factor $bb$.
Similar arguments may be used to prove the case where $k$ is the label of a cell
corresponding to the factor $aa$.
We have $rw(k) = \sum_{j=0}^{n-2} rw_j(k)$ and $lw(k) = \sum_{j=0}^{n-2} lw_j(k)$, 
 the values  $rw_j(w)$ are determined by the value $(i_2,j)$ of the coordinates of the cell situated immediately
 on the left of the North step of the Dyck word situated at row $j$. More precisely we have:
 \begin{itemize}
 \item If $ j > i_1$ then 
 $$rw_{j}(k) = Max(0,i_1-i_2) \hskip 2cm lw_{j}(k) = Max(0,i_2-i_1)$$
  \item If $ j \leq  i_1$ then 
 $$rw_{j}(k) = Max(0,i_1-i_2-1) \hskip 2cm lw_{j}(k) = Max(0,i_2-i_1+1)$$

 \end{itemize}
 But the values $i_1-i_2$ and $i_1-i_2-1$ are exactly those of heights of $\tau(w,k_i)$ increased by 1.
\end{proof}

The proof of Proposition \ref{prop:toxy}, follows directly from Lemmas
\ref{lem:fromuleCroisements}, and \ref{lem:conjuguerPourFacteurxx}.

\subsection{Sum on all cut skew cylinders of any circumferences}

Proposition~\ref{prop:toxy} is a finite and combinatorial description of $K_n(r,d)$ via an evaluation of $L_n(x,y)$ for $n\geq 2$.
In order to  make deeper connection with other chapters of combinatorics, we consider the generating function
$$ \mathbb{L}(x,y;z) =\sum_{n\geq 1} L_{n}(x,y)z^{n-1}$$
and give a simple formula for it involving two copies of the Carlitz $q$ analog series for Catalan numbers.

Before entering into the computation of this function,  we need to give a value for $L_1$.

\subsubsection*{On the complete graph $K_1$}

On $K_1$, a configuration is limited to the value $f_1$ assigned to the unique vertex.
Its degree is $f_1$ and its rank is $-1$ if $f_1 <0$ and equal to $f_1$ otherwise so that:

$$K_1(d,r) = \frac{1}{r}\frac{\frac{1}{d}}{1-\frac{1}{d}}+\frac{1}{1-rd}$$

Applying the formula
$$ K_n(d,r) = \frac{d^{{n-1 \choose 2}-1}}{r}L_n(\frac{1}{d},rd)$$

gives

$$L_1(x,y) = \frac{x}{1-x} +  \frac{1}{1-y} = H(x,y)$$
\subsubsection*{A formula for the generating function of $L_n(x,y)$}

The  generating series of the  Tutte polynomials of the complete graphs, were proved to have compact expressions (see \cite[chapter 5]{gesselSagan}, \cite[equation (17)]{tutte2}).

Notice that these polynomials enumerate  the spanning trees (in bijection with recurrent configurations) of $K_n$
using two parameters external and internal activity. 

For the series  $\mathbb{L}(x,y;z)$ can also be expressed by a compact formula
as a ``simple'' rational function of $z$ and two copies of Carlitz $q$-analogue:
$$ C(q,z) = \sum_{n\geq 0} \sum_{w\in D_n} q^{area(w)}z^n$$
where $D_n$ denotes the set of Dyck words of semi-length $n$.

\begin{theorem}
\label{thm:LnC}
We have 
$$ \mathbb{L}(x,y;z) = \frac{(1-xy)}{(1-x)(1-y)}\frac{C(x,xz)+C(y,yz)-C(x,xz)C(y,yz)}{1-C(x,xz)zC(y,yz)}.$$
\end{theorem}

\begin{proof}
$\bullet$ For $n=1$, we use the previously computed generating function $L_1(x,y)=H(x,y).$

$\bullet$ For $n\geq 2$, we use the description of $L_n(x,y)$ in Proposition~\ref{prop:toxy}.
To estimate $ \sum_{n\geq 2} L_{n}(x,y)z^{n-1}$ we have to compute the generating function of words on the alphabets $\{a,b\}$ starting and ending by \emph{different occurrences} of the same letter and containing one more occurrence of the letter $b$ than that of the letter $a$. 
Since there is a negative sign to the weight  if the first letter is $a$ and a positive sign if the first letter is $b$ we discuss separately this two cases.
We denote by $\mathbb{X}$ the language consisting of all Dyck words, including the empty word denoted by $\epsilon$.
We denote $\alpha$ the morphism  on words on alphabet $\{a,b\}$ exchanging the letters $a$ and $b$, so that $\alpha(a)=b$ and $\alpha(b)=a$ and denote $\mathbb{Y}=\alpha(\mathbb{X})$.
 
Let $f=bf''b$ be a word containing $n$ occurrences of the letter $b$,  $n-1$ occurrences of the letter $a$ (and such that 
the first and the last letters are occurrences of $b$).

This word may be interpreted as the sequence of steps of a walk on the line $\mathbb{Z}$, a letter $a$ corresponding to an increment, and a letter $b$ to a decrement. We assume that this walk starts at $1$ and hence ends at $0$ and decompose it at each step from $1$ to $0$ or $0$ to $1$. In terms of language, this decomposition leads to the following non-ambiguous description of these words:
$$ b\mathbb{Y}a(\mathbb{X}b\mathbb{Y}a)^*\mathbb{X}b.$$ 
The weight $W[f](x,y)$ is distributed on the occurrences of the letter $a$ as in its definition. We add a factor $z$ to each occurrence of the letter $a$ in order to take into account the size of words.
Each factor $\mathbb{X}$ corresponds to a ``Dyck'' walk starting and ending at $1$, so compared to the usual definition, the height of a step is incremented by one. Hence the generating function of this factor is $C(x,xz)$.
Similarly, each factor $\mathbb{Y}$ has generating function $C(y,yz)$.
Each $a$ step from $0$ to $1$ is weighted by $z$ so the previous language decomposition leads to the generating function
$$ \frac{C(x,xz)zC(y,yz)}{1-C(x,xz)zC(y,yz)}.$$ 

We repeat a similar proof for words whose extremal letters are two different occurrences of the  letter $a$.
The interpretation in terms of walk on the line and its decomposition according to the same steps  leads to the formula
$$ (\mathbb{X}\backslash \{\epsilon\})b(\mathbb{Y}a\mathbb{X}b)^*(\mathbb{Y}\backslash\{ \epsilon\})$$
where at some steps we had to exclude the empty word to guarantee the extremal occurrences of letter $a$.
This leads to the generating function for this case
$$ \frac{(C(x,xz)-1)(C(y,yz)-1)}{1-C(x,xz)zC(y,yz)}.$$ 

Hence, the difference of the two generating functions gives
$$ \sum_{n\geq 2} L_{n}(x,y)z^{n-1} = \frac{(1-xy)}{(1-x)(1-y)}\frac{C(x,xz)zC(y,yz) - (C(x,xz)-1)(C(y,yz)-1)}{1-C(x,xz)zC(y,yz)}.$$
This formula added to the case $n=1$ ends the proof of our Theorem.
\end{proof}

\newpage
\appendix
\section{An alternative using pointed cut skew cylinders}
In the main part of this paper the proofs of the main results concerning the rank in $K_n$ are given using a classical object, namely the  
set of Dyck words. We wish to introduce in this appendix another combinatorial object which gives a wider geometric insight and guided some of our intuitions on significant parts of this work.
In addition, this setting seems also adapted to the similar study on the bipartite complete graphs $K_{m,n}$.
This case is out of the scope of this paper but we have work in progress about it.
Hence, in this appendix we present quickly these objects, skipping potentially tedious details by pointing to the related detailed proofs in the paper.
A consequence of this lake of detailed proof is that propositions are there called statements, even if we estimate that the reader can figure out the details from our preceeding proofs to avoid such a drop in status. 

\subsection{Some operators on compact sorted configurations}
We deal only with the complete graph $K_n$ for some fixed $n\geq 2$.
Let $f\in\mathbb{Z}^n$ and $g\in \mathbb{Z}^n$ be generic configurations on $K_n$, we recall that we fix the vertex $n$ to be the sink 
used to define the  parking configurations.
The toppling equivalence relation $f\sim_{L_G} g$ induces classes.

We represent the class containing the configuration $f$ by the existing and unique parking configuration denoted $\mathrm{park}(f)$,which is toppling equivalent to $f$.
We also consider the ($S_{n-1}$)-permuting equivalence, $f\sim_{S_{n-1}} g$ if $f_n=g_n$ and there exists a permutation $\sigma\in S_{n-1}$ on the elements $\{1,\ldots n-1\}$ such that $f_i=g_{\sigma(i)}$ also denoted $f=\sigma.g$. 
We represent the class containing the configuration $f$ by the existing and unique sorted configuration denoted $\mathrm{sort}(f)=\sigma.f=g$ where $\sigma$ is a sorting permutation in $S_{n-1}$ leading to $g$ whose first $n-1$ entries $g_1\leq g_2 \leq \ldots \leq g_{n-1}$ are weakly increasing and $f_n=g_n$. 
Hence any configuration $f$ can be represented by the sorted configuration $\mathrm{sort}(f)$ and some permutation $\sigma\in S_{n-1}$, where $\sigma$ is not necessarily unique. 
These two distinct toppling and permuting equivalences leads to the toppling and permuting equivalence: $f\sim_{L_{G},S_{n-1}} g$ if there exists $\sigma\in S_{n-1}$ such that $f\sim_{L_G} \sigma.g$, notice that in that case the ranks and degrees of the configurations $f$ and $g$ are equal.
Since  sorting  the $n-1$ first entries, preserves the property of being a parking configuration, the toppling and permuting class of a configuration $f$ is represented by the existing and unique sorted and parking configuration $\mathrm{sort}\circ\mathrm{park}(f)$  which is toppling and permuting equivalent to $f$.

Most of the computations for the rank on $K_n$ may manipulate only sorted configurations, via a $\mathrm{sort}$ after each elementary step.
\begin{definition}
A configuration $f$ is called \emph{compact} if any two of the first $n-1$ entries differ by at most $n$, that is:
$\max_{1\leq i,j \leq n-1}|f_i-f_j|\leq n.$
\end{definition}
If  sorted parking configurations are motre precise  to state our results, it appears that the weaker  notion of sorted compact configuration is sufficient to describe configurations appearing in most of our computations in our proofs for $K_n$.

First, for any configuration $f$, we denote by $\mathrm{compact}(f)$, the result $g$ of the algorithm in Section~\ref{subsec:algo1} defined for $i< n$ by $g_i=f_i-f_1 (\mbox{mod } n)$ and $g_n=deg(f)-\sum_{i=1}^{n-1}g_i$.
This map is the initial projection of the configuration $f$ into a compact configuration $\mathrm{compact}(f)$. 
Then the computations manipulate exclusively compact configurations.
\begin{definition}
We define the operator $T$ on sorted configurations by 

\centerline{$\displaystyle T(f) = \mathrm{sort}(f-\Delta^{(n-1)})$} 

Hence $T(f)$ is obtained from $f$ by toppling  the vertex $n-1$ of maximal value among the $\{f_i\}_{i=1\ldots n-1}$ and then sorting  the resulting configuration.
\end{definition}
A key property is that this operator becomes invertible when restricted to compact sorted configurations.
Indeed, explicit computation shows that for a compact sorted configuration $f$  we have:

\centerline{$\displaystyle  T((f_1,\ldots f_{n-1},f_n))=(f_{n-1}-(n-1),f_1+1,f_2+1,\ldots,f_{n-2}+1,f_n+1)$}
Since  the compact assumption in a sorted configuration $g$ is equivalent to 
$g_{n-1}\leq g_1+n$.

Hence  the permutation sorting into a parking configuration  this case is  the cyclic permutation $\tau\in S_{n-1}$ defined by $\tau(i)=i+1$ for $i<n-1$ and $\tau(n-1)=1$. 
In addition, when restricted to sorted compact configurations, the inverse of the operator is explicitely given by 

\centerline{$\displaystyle  T^{-1}((f_1,\ldots f_{n-1},f_n))=(f_{2}-1,\ldots,f_{n-2}-1,f_{n-1}-1,f_{1}+(n-1),f_n-1)$}

Moreover for a compact sorted configuration $f$, the $(n-1)$-th power of $T^{-1}$ corresponds to the toppling of the sink:

\centerline{$T^{1-n}(f)=f+\Delta^{(n)}.$}

Indeed one can show that 

\begin{statement}
For any sorted compact configuration $f$ there exists some power $i\in\mathbb{Z}$ such that 

\centerline{$\mathrm{sort}\circ\mathrm{park}(f)=T^{i}(f).$}

\end{statement}

A key point in this proof is that in a stable sorted compact configuration $f$, a subset $Y$ of maximal cardinality which is a counter example to the parking assumption is described as $\{n-1,n-2,\ldots n-k\}$, hence the toppling of all the vertices of this set are described by $T^k$. One can then conclude via the following naive algorithm computing $\mathrm{sort}\circ\mathrm{park}(f)$ almost according to the element of its definition for a compact sorted configuration $f$ :
\begin{itemize}
\item While at least one of the $n-1$ first entries is negative, topple the sink (this is equivalent to do the inverse of $T^{n-1}$).
\item While there is a subset $Y$ giving a counter example to parking assumption, pick the subset $Y$ of maximal cardinality and topple its vertices. ($T^{|Y|}$ at each loop iteration) 
\end{itemize} 
The unicity of a sorted parking configuration in a toppling and permuting equivalent class allows to reformulate this result in:

\begin{statement}
The set of sorted compact configurations toppling and permuting equivalent to the sorted compact configuration $f$ is exactly $\{T^i(f)\}_{i\in\mathbb{Z}}$ where all these powers of $T$ are distinct.
\end{statement}

All these powers are distinct since $T$ increments exactly by $1$ the value on the sink $n$.

\subsection{Pointed cut skew cylinders to represent previous operators} 

The doubly pointed cut skew cylinders defined below appear to be the natural combinatorial objects which describe $\{T^i(\mathrm{sort}\circ\mathrm{park}\circ\mathrm{compact}(f))\}_{i\in\mathbb{Z}}$, they  also describe the  graphical description of our algorithm computing the rank on $K_n$ and later the analysis of this algorithm. 

In the following we also call vertex a point of the plane $\Z^2$.
\begin{definition}The \emph{skew cylinder $C_n$} of circumference $2n-1$ is obtained from the usual two dimensional grid $\mathbb{Z}^2$ by identifying two vertices of coordinates $(i,j)$ and $(i',j')$ if and only if $(i',j') = (i,j)+k(n,n-1)$ for some $k\in\mathbb{Z}$ also denoted $(i,j)\sim_{(n,n-1)} (i',j')$.
\end{definition}
The skew cylinder may be described as illustrated in Figure~\ref{fig:strip} by a strip made up of $n-1$ rows, each row containing an infinite number of cells. In this Figure we have added two virtual copies of rows above and below the strip, these express the neighborhood in the skew cylinder for each cell in the top and bottom row of the strip. Notice that traversing the cylinder from one cell to the one at the north-east of it, defines a unique {\em diagonal} which visits all the cells of the strip.

This kind of skew cylinder is far from new since it already appeared or example in 1941 in a work of Kramers and Wannier \cite{kramersWannier}.

The unique north-east diagonal allows to label the vertices (i.e points in $ \Z^2$ )as follows:
\begin{definition} (of labeling and spiral traversal of vertices in a skew cylinder)
Label  $0$ the origin of the skew cylinder and then from the vertex labeled  $i$, the vertices of label $i+1$, and  $i-1$ are obtained by a diagonal north-east step $(+1,+1)$ for the first and  a south-west step $(-1,-1)$ for the second.
Hence, the visit of vertices in increasing order according to the label in $\mathbb{Z}$ describes what we call a \emph{spiral traversal} of the vertices.
\end{definition}
Notice that this spiral traversal corresponds to a visit of the single diagonal of the skew cylinder in north-east direction.

We will show how it is  related both to the powers $T^{i}(\mathrm{sort}\circ\mathrm{park}(f))$ and to the order of labeling in our algorithm computing the rank on $K_n$.

\begin{figure}[ht!]
\begin{center}
\begin{tikzpicture}[scale=0.98]
\foreach \x/\y/\celllabel in {6/4/-4,5/4/0,4/4/4,3/4/8,2/4/12,1/4/16}{
\draw[line width=2,draw=black!40!white] (\x-1,\y) rectangle (\x,\y+1);
\draw node[black!40!white] at (\x-0.5,\y+0.5) {$\celllabel$};
}
\foreach \x/\y/\celllabel in {0/-1/-5,-1/-1/-1,-2/-1/3,-3/-1/7,-4/-1/11,-5/-1/15}{
\draw[line width=2,draw=black!40!white] (\x-1,\y) rectangle (\x,\y+1);
\draw node[black!40!white] at (\x-0.5,\y+0.5) {$\celllabel$};
}
\foreach \x/\y/\celllabel in {0/0/0,1/1/1,2/2/2,3/3/3,-1/0/4,0/1/5,1/2/6,2/3/7,-2/0/8,-1/1/9,0/2/10,1/3/11,-3/0/12,-2/1/13,-1/2/14,0/3/15,-4/0/16,-3/1/17,1/0/-4,2/1/-3,3/2/-2,4/3/-1,5/3/-5,4/2/-6}{
\draw[line width=2] (\x-1,\y) rectangle (\x,\y+1);
\draw node at (\x-0.5,\y+0.5) {$\celllabel$};
}
\draw node[rotate=45] at (-2.5,2.5) {$\ldots$};
\draw node[rotate=45] at (2.5,1.5) {$\ldots$};
\draw (-6,0) grid (6,4);
\end{tikzpicture}
\caption{In black, the skew cylinder $C_4$ as a strip. The label of a cell is the label of its bottom right corner in the spiral traversal. The wrapping conditions are described by the grey rows which are translated copies of the top and bottom rows. Notice that starting from the vertex labeled by $0$ on the bottom (black) row one reaches its copy on the top (grey) row after any combination of $4$ north steps and $5$ east steps and such a path disconnects the cylinder  in two parts.}\label{fig:strip}
\end{center}
\end{figure}
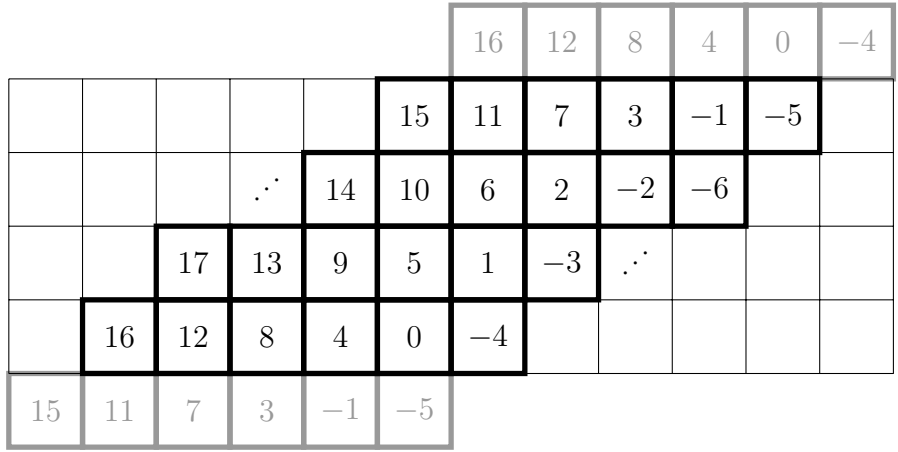

This cylinder $C_n$ is said of circumference $2n-1$ since   disconnecting it by a general cut, which is a self-avoiding loop starting from the origin and made up of north or east steps, requires $2n-1$ elementary steps: $n-1$ north steps and $n$ east steps.
If one maps a north step to the letter $a$ and an east step to the letter $b$, it obtains an obvious bijection between general cuts of $C_n$ and the words of $A_n$.
We will impose that the cut is described by a word $w$ of $D_n\subsetneq A_n$, these cuts are in bijection with the set of all  $n-1$ first entries of the sorted parking configurations on $K_n$. 
For $w \in D_n$, we denote $C_n[w]$ the skew cylinder disconnected by the cut described by the word/path $w$ starting from the origin.
\begin{definition}
A \emph{doubly pointed cut skew cylinder} $C_n[w](s,x)$ is defined by a triplet $(w,s,x)\in D_n\times\mathbb{Z}\times\mathbb{Z}$ where $s$ and $x$  describe a pointer on the vertices with  labels $s$ and respectively $x$ in the cut skew cylinder $C_n[w]$.
\end{definition}
 The two pointers may be equal ($s=x$) but if $s\neq x$ then $C_n[w](s,x)$ and $C_n[w](x,s)$ are distinct.

The following explicit mapping  $\mathrm{cyltoconf}$ is a bijection between the set of  doubly pointed cut  skew cylinders and 
 the  infinite set of compact sorted configurations on $K_n$. 

\begin{definition}
For any doubly pointed cut skew cylinder $C[w](s,x)$, the configuration $$f=\mathrm{cyltoconf}(C[w](s,x))$$ is defined by $f_n=s$ and for $i<n$,$$f_i=\frac{1}{n-1}(x_i-d_i)$$
where $x_i=x+n(i-1)$ and $d_i$ is the label of the vertex which is the start of a north step of the cut $w$ and on the same row as the vertex of label $x_i$.
\end{definition}
This definition has a   graphical interpretation by drawing  a segment of $n-1$ north steps starting from the  vertex labeled $x$, 
the values of the $f_i$'s   are equal to  the relative position of a  north steps of the   segment and the north step of the cut $w$ in the same row.   

\begin{figure}[ht!]
\begin{center}
\begin{tikzpicture}[scale=0.8]
\draw (-8,0) grid (9,5);
\draw[line width=2] (4.5,2.5) circle (0.35);
\draw node at (4.5,2.5) {$s$};
\draw[line width=2,draw=red,->] (-2,0) -- (-2,5);
\foreach \y in {1,2,...,5}{
\draw node at (-2.5,\y-0.5) {$x_\y$};
}
\draw[draw=red,line width=2] (-2.5,0.5) circle (0.35);
\draw[line width=3,draw=blue] (0,0) -- (0,1) -- (1,1) -- (1,3) -- (2,3) -- (2,4) -- (4,4) -- (4,5) -- (6,5);
\foreach \x/\y in {0/1,1/2,1/3,2/4,4/5}{
\draw node at (-0.5+\x,\y-0.5) {$c_\y$};
}
\draw[draw=blue,line width=2] (-0.5,0.5) circle (0.35);
\foreach \y in {1,2}{
\draw node at (-0.5,\y+3-0.5) {$x'_\y$};
}
\foreach \y in {3,4,5}{
\draw node at (-6.5,\y-2-0.5) {$x'_\y$};
}
\draw[draw=green,line width=2] (-0.5,3.5) circle (0.35);
\draw[draw=green,line width=2,->] (0,3) -- (0,5);
\draw[draw=green,line width=2,->] (-6,0) -- (-6,3);
\foreach \y in {1,2,...,5}{
\draw node at (7.5,\y-0.5) {$x''_\y$};
}
\draw[line width=2,draw=orange,->] (8,0) -- (8,5);
\draw[draw=orange,line width=2] (7.5,0.5) circle (0.35);
\draw[line width=1,dashed,draw=blue] (0,0) -- (5,5);
\end{tikzpicture}
\end{center}
\caption{\label{fig:defcompact}}
\end{figure}

Figure~\ref{fig:defcompact} contains three examples of $\mathrm{cyltoconf}$ for the cut $w=abaababbabb\in D_5$ drawn in blue.
The dashed blue line helps to check that $w\in D_5$ showing the factorization $w=w'b$ where $w'$ is a Dyck word.
The label in a cell is the label of its bottom right corner.
\begin{itemize}
\item $\mathrm{cyltoconf}(w,s,x_1)=(2,3,3,4,6,s)$  (defined by cells pointed by black and red circles in $\mathrm{Cyl}[w]$)
\item $\mathrm{cyltoconf}(w,s,x'_1)=(2,4,6,7,7,s)$ (defined by cells pointed by black and green circles in $\mathrm{Cyl}[w]$)
\item $\mathrm{cyltoconf}(w,s,x''_1)=(-8,-7,-7,-6,-4,s)$ (defined by cells pointed by black and orange circles in $\mathrm{Cyl}[w]$).
\end{itemize} 
By convention, the cell containing $c_1$ is labeled by $0$, then $s=-13$, $x_1=10$, $x'_1=18$ and $x''_1=-40$.

The motivation for introducing the map $\mathrm{cyltoconf}$ comes from the following description of the  operators on sorted compact configurations.

\begin{statement}
For any sorted parking configuration $f$, we have
$$\mathrm{sort}\circ\mathrm{park}(f)=\mathrm{cyltoconf}((C_n[\phi(f)],\psi(f),0)).$$
Moreover, for any doubly pointed cut skew cylinder $(C_n[w],x,s)$ we have
$$T(\mathrm{cyltoconf}((C_n[w],s,x))) = \mathrm{cyltoconf}((C_n[w],s+1,x-1)),$$
$$\mathrm{cyltoconf}((C_n[w],s,x))-\Delta^{(n)}=\mathrm{cyltoconf}((C_n[w],s-(n-1),x+(n-1))),$$
$$\mathrm{cyltoconf}((C_n[w],s,x))+\epsilon^{(n)}=\mathrm{cyltoconf}((C_n[w],s+1,x)),$$
$$\mathrm{sort}(-\mathrm{cyltoconf}((C_n[w],s,x)))=\mathrm{cyltoconf}((C_n[\Phi(w)],-s,\mathrm{lastright}(w)-x-k)).$$ 
Where $k = n(n-3)+1$.
\end{statement}

The end of this section is devoted to give  some elements of proof for this statement.
Adding a final east step to the (blue) Dyck path in Figure~\ref{fig:calcRank3} leads to recognize the cut of a skew cylinder starting at the origin which is the bottom right corner of cell labeled by $0$.

This example should convince that for any sorted parking configuration $f$ we have 

\centerline{$f=\mathrm{park}(f)=\mathrm{cyltoconf}((C_n[\phi(f)],\psi(f),0))$}

 where the north steps used as $(x_i)_{i=1\ldots n-1}$ are the north steps on the $y$-axis.
In addition, inspection shows that for any doubly pointed cut skew cylinder we have

\centerline{$T(\mathrm{cyltoconf}((C_n[w],s,x)) = \mathrm{cyltoconf}((C_n[w],s+1,x-1))$.}

These two first remarks are combined to show that $\mathrm{cyltoconf}$ is a bijection from the doubly pointed cut skew cylinders of circumference $2n-1$ to the sorted compact configurations described as

\centerline{$\{T^i(f)| \mbox{ $f$ sorted parking configuration  on $K_n$, $i\in\mathbb{Z}$}\}$.}

Iterating the description of (the inverse of) $T$ we obtain 

\centerline{$\mathrm{cyltoconf}((C_n[w],s,x))-\Delta^{(n)}=\mathrm{cyltoconf}((C_n[w],s-(n-1),x+(n-1)))$,}

and obviously the increment by one of the value on the sink is described by 

\centerline{$\mathrm{cyltoconf}((C_n[w],s,x))+\epsilon^{(n)}=\mathrm{cyltoconf}((C_n[w],s+1,x))$.}

A less obvious relation describing the operator $f\longrightarrow \mathrm{sort}(-f)$ is a key to understand globally many properties of involutions of this paper and explain the origin and name of the parameter $\mathrm{lastright}(w)$:

\centerline{$\mathrm{sort}(-\mathrm{cyltoconf}((C_n[w],s,x)))=\mathrm{cyltoconf}((C_n[\Phi(w)],-s,\mathrm{lastright}(w)-x-k))$.}

The proof of this relation is the subject of the three following paragraphs.
We define two symmetries on skew cylinders. 
Then we show  a link by graphic superimposition between $C_n[w]$ and $C_n[\Phi(w)]$ described by the preceding symmetries.
Finally we refine this superimposition to the definitions of $f$ and $\mathrm{sort}(-f)$ by $\mathrm{cyltoconf}$ to obtain the expected relation. 

First we define two symmetries on skew cylinders involved in a link between $C_n[w]$ and $C_n[\Phi(w)]$.

\begin{definition}(shift and flip on skew cylinders)

The \emph{unit diagonal shift} on $\mathbb{Z}^2$ is defined by  $\mathrm{shift}((i,j))=(i+1,j+1)$.
This symmetry is compatible with the equivalence $\sim_{(n,n-1)}$ defining $C_n$ hence induces 
a shift with similar notations on skew cylinder $C_n$.

The \emph{flip symmetry} $\mathrm{flip}((i,j))=(-i,-j)$ on the two dimensional grid $\mathbb{Z}^2$ is also compatible with the relation $\sim_{(n,n-1)}$ hence induces a flip with similar notations on skew cylinder $C_n$.
\end{definition}

The shift on skew cylinder maps each vertex to the next according to the spiral traversal.
The compatibility of the flip with $\sim_{(n,n-1)}$ comes from  a multiplication by $-1$ of the relation assuming  $(i,j) \sim_{(n,n-1)} (i',j')$ which leads to a one  proving $(-i,-j) \sim_{(n,n-1)} (-i',-j')$.
We notice that the image $F(C_n)$ by the flip symmetry of the skew cylinder $C_n$ of circumference $2n-1$ is $C_n$ except that the spiral traversal of vertices is reversed since the vertex labeled by $x$ becoming labeled  $-x$.

\begin{center}
\begin{figure}[ht!]
\begin{tikzpicture}
\draw (-6,0) grid (9,7);
\draw[line width=2,draw=blue] (0,0) -- (0,2) -- (1,2) -- (1,5) -- (3,5) -- (3,6) -- (5,6) -- (5,7) -- (8,7);
\draw[line width=1,dashed,draw=green] (0,0) -- (7,7);
\draw[line width=0,fill=green] (0,0) circle (0.1);
\draw[line width=1,dashed,draw=red] (1,5) -- (-4,0);
\draw[line width=1,dashed,draw=red] (4,7) -- (2,5);
\draw[line width=0,fill=red] (1,5) circle (0.1);
\foreach \i/\j in {1/1,2/2,3/3,4/4,5/5,6/6,2/3,3/4,4/5,2/4}{
\draw[line width=0,fill=green,opacity=0.5] (\i-0.45,\j+0.55 ) rectangle (\i-0.05,\j+0.95);
}
\foreach \i/\j in {-3/1,-2/1,-1/1,-2/2,-1/2,-1/3,0/3,0/4,4/7}{
\draw[line width=0,fill=red,opacity=0.5] (\i+0.05,\j-0.95 ) rectangle (\i+0.45,\j-0.55);
}
\draw[draw=green,line width=2,->] (-5.05,0) -- (-5.05,7);
\draw[draw=red,line width=2,->] (-4.95,7) -- (-4.95,0);
\foreach \i/\j/\wlabel in {0/0/0,1/1/1,2/2/2,3/3/3,4/4/4,5/5/5,6/6/6,-1/0/7,0/1/8,1/2/9,2/3/10,3/4/11,4/5/12,5/6/13,-2/0/14,-1/1/15,0/2/16,1/3/17,2/4/18,3/5/19,4/6/20,-3/0/21,-2/1/22,-1/2/23,0/3/24,1/4/25,2/5/26,3/6/27,-4/0/28,-3/1/29,-2/2/30,-1/3/31,0/4/32,1/5/33,-5/0/35,-5/1/43,-5/2/51,-5/3/59,-5/4/67,-5/5/75,-5/6/83}{
\draw node[green] at (\i-0.25,\j+0.25) {\tiny{$\wlabel$}};
}
\foreach \i/\j/\phiwlabel in {1/5/0,0/4/1,-1/3/2,-2/2/3,-3/1/4,4/7/5,3/6/6,2/5/7,1/4/8,0/3/9,-1/2/10,-2/1/11,5/7/12,4/6/13,3/5/14,2/4/15,1/3/16,0/2/17,-1/1/18,6/7/19,5/6/20,4/5/21,3/4/22,2/3/23,1/2/24,0/1/25,7/7/26,6/6/27,5/5/28,4/4/29,3/3/30,2/2/31,1/1/32,8/7/33,-5/1/-10,-5/2/-18,-5/3/-26,-5/4/-34,-5/5/-42,-5/6/-50,-5/7/-58}{
\draw node[red] at (\i+0.25,\j-0.25) {\tiny{$\phiwlabel$}};
}
\end{tikzpicture}
\caption{Superimposition of $\mathrm{cyltoconf}((C_8[aabaaabbabbabbb],s,35))$ and $\mathrm{shift}^{33}\circ\mathrm{flip}(\mathrm{cyltoconf}((C_8[aaabaabbbabbabb],-s,-58)))$\label{fig:superimposition}}
\end{figure}
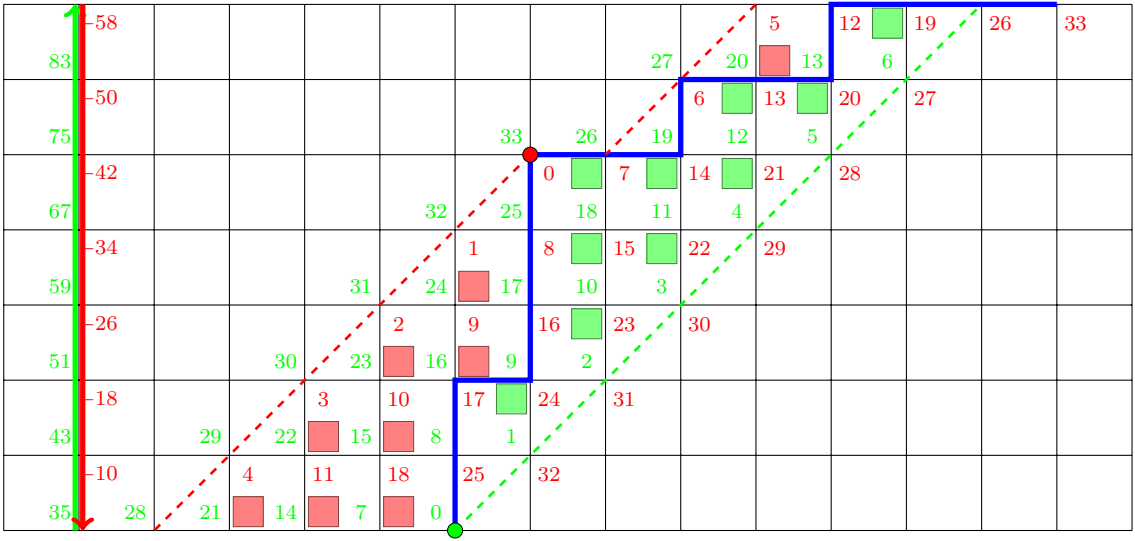
\end{center}

A key element of proof is the superimposition of the graphic descriptions of most notions in $C_n[w]$ and $\mathrm{shift}^{\mathrm{lastright(w)}+2(n-1)}\circ \mathrm{flip}(C_n[\Phi(w)])$.
Figure~\ref{fig:superimposition} provides an example of such a superimposition.
There, each vertex $x$ has two labels, the (green) label $[x]_w$ in $C_n[w]$ drawn, as previously, in the bottom right corner of the top left cell incident to the vertex and the (red) label $[x]_{\Phi(w)}$ in $\mathrm{shift}^{\mathrm{lastright}(w)+2(n-1)}\circ \mathrm{flip}(C_n(\Phi(w)))$ drawn in the top left corner of the bottom right cell incident to the vertex.
To prove such superimposition, one remarks that the image of the cut $w\in D_n$ by $\mathrm{flip}$ is the general cut $\tilde{w}\in A_n$.
The shift is then iterated so that the conjugate $\Phi(w)$ of $\tilde{w}$ in $D_n$ starts at the origin.
It means that  the vertex of label $[0]_{\Phi(w)}$ is also the vertex reached after reading the prefix $u$ in the factorisation $w=uv$ 
such that $u$ is the longest prefix with maximal value by $\delta$. 
Hence the vertex labeled by $[0]_{\Phi(w)}$ is also labeled in $C_n[w]$ by
$$ [n|u|_a-(n-1)|u|_b]_w=[\mathrm{lastright}(w)+2n-1]_w,$$ 
where may be defined by $\mathrm{lastright}(w)=n|u|_a-(n-1)(|u|_b+1)$.
On the example in Figure~\ref{fig:superimposition}, $n=8$, $|u|_a=4$, $|u|_b=1$, $\mathrm{lastright}(w)=8\times 4-7\times(1+1)=18$, $\mathrm{lastright}(w)+2n-1 = 18+2\times 8 - 1 = 33$ and the vertex labeled by $[0]_{\Phi(w)}$ is also labeled by $[33]_{w}$. 
This remark helps to determine the number of iterations of $\mathrm{shift}$ for the superimposition:
the flip maps the vertex $[\mathrm{lastright}(w)+2n-1]_{w}$ in $C_n[w]$ to the vertex $[-\left(\mathrm{lastright}(w)+2n-1\right)]_{\tilde{w}}$ in $C_n[\tilde{w}]$, then $(\mathrm{lastright}(w)+2n-1)$ shift maps this vertex to $[0]_{\Phi(w)}$ in $C_n[\Phi(w)]$.

This superimposition is then refined to $f=\mathrm{cyltoconf}((C_n[w],s,x)))$ and $\mathrm{sort}(-f)=\mathrm{cyltoconf}((C_n[\Phi(w)],-s,\mathrm{lastright}(w)-x-k)).$
Indeed, in the definition of $f$ by $\mathrm{cyltoconf}$ the segment of $n-1$ north steps starting from vertex labeled $[x]_w$ and ending in vertex $[x+n(n-1)]_w$ in $C_n[w]$ is superimposed, up to reverse, to the image of the similar segment used in the definition of some configuration $g$ in $C_n[\Phi(w)]$.
It appears that $g=\mathrm{sort}(-f)$. 

Indeed, first the permutation used in this case of $\mathrm{sort}$ is $\omega \in S_{n-1}$ defined by $\omega(i)=n-i$ for $i<n$.
This matches the reverse of segment induced by the $\mathrm{flip}$ symmetry.
Then each row we observe that $g_i=-f_{n-i}$ for some $i<n$ since both definitions use the same pair of north steps from the segment and the cut but there relative position change of sign due to the $\mathrm{flip}$ symmetry.
The case of the value on the sink does not corresponds to a superimposition but the map $s\longrightarrow -s$ obviously satisfies the constraint $g_n=-f_n$. 
We obtain the expected relation by observing that in the superimposition described using $\mathrm{shift}^{\mathrm{lastright}(w)+2n-1}\circ \mathrm{flip}$, the vertex labeled by $[x+n(n-1)]_w$ in $C_n[w]$ is labeled by $$[(\mathrm{lastright}(w)+2n-1)-x-n(n-1)]_{\Phi(w)}=[\mathrm{lastright}(w)-k]_{\Phi(w)}$$ in $C_n[\Phi(w)]$. 
On the example in Figure~\ref{fig:superimposition}, we have $n=8$, $x=35$, the (green) segment of $n-1$ north steps defining $\mathrm{cyltoconf}(C_8[aabaaabbabbabbb],s,35)$ starts in $[35]_w$ and ends in $[91]_w=[35+7\times 8]_w$.
The (red) segment of $n-1$ south steps defining $\mathrm{cyltoconf}(C_8[aaabaabbbabbabb],-s,-58)$ starts in $[33-91]_{\Phi(w)}=[-58]_{\Phi(w)}$. 
\subsection{Revisiting some properties of discussed involutions}

The preceeding superimposition gives an alternative setting to describe the involution $\Phi$.
Since the shift and flip maps act on vertices, they can act on steps seens as a pair of vertices.
Using the identification between general cuts and paths of $A_n$, we may describe $\Phi$ as follows:

\begin{statement}
Let $w \in D_n$, we have 
$$ \Phi(w) = \mathrm{shift}^{\mathrm{lastright}(w)+2n-1}\circ \mathrm{flip}(w).$$ 
\end{statement}

\subsubsection*{$\Phi$ and bistatistic $(\mathrm{area},\mathrm{prerank})$}
The superimposition also helps to see at the level of rows of the skew cylinder why the involution $\Phi$ exchanges the $\mathrm{prerank}$ and $\mathrm{area}$ statistics.
Here we denote by $\mathrm{height}_w(v)$ the height of the vertical step $v$ in the cut $w$ which is also defined as the number of cells on the row crossed by $v$ which are completly between this step and the (green dashed) diagonal step crossing this row in Figure~\ref{fig:superimposition}.
These cells are marked by a green square.
We denote by $\mathrm{coheight}_w(v)$ the coheight of the vertical step $v$ in the cut $w$ which appears to be defined as the number of cells on the row crossed by $v$ which are completly between this step and the (red dashed) diagoanl step crossing this row in Figure~\ref{fig:superimposition}.
For example, in Figure~\ref{fig:superimposition}, for each rows from bottom to top the heights  are $(0,1,1,2,3,2,1)$ and the coheights are $(3,2,2,1,0,0,1)$.  
Inspection shows that these two definitions coincide with the definitions given in the main part of this article.
Then, the superimposition exchanges the meaning of the red and green dashed diagonals hence make obvious the following statement 
\begin{statement}
Let $w\in D_n$.
For any vertical step $v$ of the superimposed cuts $w$ and $\Phi(w)$ in the superimposition of $C_n[w]$ and $\mathrm{shift}^{\mathrm{lastright}(w)+2n-1}\circ \mathrm{flip}(C_n[\Phi(w)])$,
$$ \mathrm{height}_{w}(v)=\mathrm{coheight}_{\Phi(w)}(v) \mbox{ and } \mathrm{coheight}_{w}(v)=\mathrm{height}_{\Phi(w)}(v) $$ 
\end{statement}
Summing this statement on all rows implies that the area, which is the sum of heights, and the prerank, which is the sum of coheights are exchanged by $\Phi$.

\subsubsection*{$\Phi$ preserves $\mathrm{dinv}$ parameter}.
We introduce an alternative definition of $\mathrm{dinv}$ statistic called $\mathrm{cdinv}$ for cyclic $\mathrm{dinv}$ which is more obviously preserved by $\Phi$.
\begin{definition}
Let $w\in D_n$.
We consider on each row $i$ of $C_n[w]$, from bottom to top, the contact $c_i$ which is the label in $C_n[w]$ of the vertex where start the north step crossing this row $i$. 
Then
$$\mathrm{cdinv}(w) = |\{ \{i,j\}| 1\leq i<j \leq n-1 \mbox{ and } |c_i-c_j| \leq n-1|\}|$$
\end{definition}

On the example in Figure~\ref{fig:superimposition}, we have
$$(c_i)_{i=1\ldots n-1} = 0,8,9,17,25,19,13$$
and 
$$\mathrm{cdinv}(w) = |\left\{ \{8,9\}, \{8,13\},\{9,13\},\{13,17\},\{13,19\},\{17,19\},\{19,25\}\right\}|=7.$$

This definition is motivated by the following statement:
\begin{statement}
Let $w=w'b\in D_n$, then $$\mathrm{dinv}(w')=\mathrm{cdinv}(w)$$ and $$\mathrm{cdinv}(\Phi(w))=\mathrm{cdinv}(w).$$
\end{statement}

The preceeding example for the definition of $\mathrm{cdinv}$ helps to convince the reader that this one case definition of pairs for $\mathrm{cdinv}(w)$ coincide with the two case definition of $\mathrm{dinv}(w')$.
More precisely, the vertices of label $[c_i]_w$ and $[c_i+n-1]_w$ are on the same row, and for any pair $\{c_i,c_j\}$ counted for $\mathrm{cdinv}$ , with $c_i<c_j$, the two cases in the definition of $\mathrm{dinv}$ corresponds to the discussion if the row of the vertex $0$ belongs to one of the row of $[c_i+k]_w$ for some $k=0,1,\ldots c_j-c_i$.
Hence $\mathrm{cdinv}(w)=\mathrm{dinv}(w')$.

The contacts are defined as the starts of vertical steps, and up to orientation, these steps are preserved in the superimposition of $C_n[w]$ and $\mathrm{shift}^{\mathrm{lastright}(w)+2n-1}\circ \mathrm{flip}(C_n[\Phi(w)])$.
In this setting, the contacts of $\Phi(w)$ are the opposite vertices to the vertical of the contacts of $w$.
More precisely, starting from the contact of label $[c_i]_w$, the opposite vertex to the related north step is of label $[c_i+n]_w$ in $C_n[w]$, equivalently of label $[\mathrm{lastright}(w)+2n-1-(c_i+n)]_{\Phi(w)}$ in $C_n[\Phi(w)]$.
Hence the images of the contacts $([c_i]_w)_{i=1\ldots n-1}$ are the contacts $\{c'_i\}_{i=1\ldots n-1}=\{[\mathrm{lastright}(w)+n-1-c_i]_{\Phi(w)}\}_{i=1\ldots n-1}$.
We notice that $c'_i-c'_j=-(c_i-c_j)$ so the map $\{c_i,c_j\}\longrightarrow \{c'_i,c'_j\}$ defines a bijection between the  pairs counted for $\mathrm{cdinv}(w)$ and those counted for $\mathrm{cdinv}(\Phi(w))$, showing that $\mathrm{cdinv}(w)=\mathrm{cdinv}(\Phi(w))$.

On the example in Figure~\ref{fig:superimposition}, we have the contact of $\Phi(w)$ which are 
$$(c'_i)_{i=1\ldots n-1} = 0,8,16,17,25,12,6$$
and
$$\mathrm{cdinv}(\Phi(w)) = |\left\{ \{0,6\}, \{6,8\},\{6,12\},\{8,12\},\{12,16\},\{12,17\},\{16,17\}\right\}|=7.$$
For example the pair $\{0,6\}$ in the computation of $\mathrm{cdinv}(\Phi(w))$ corresponds to the pair $\{19,25\}$ in the computation of $\mathrm{cdinv}(w)$.

\subsubsection*{$\Psi$ via the operators related to $\mathrm{cyltoconf}$}

The operators on compact sorted configurations that we describe via the map $\mathrm{cyltoconf}$ allows a derivation via this setting of a preceeding explicit description of the involution $\Psi(f)=\mathrm{sort}\circ \mathrm{park}(\kappa-f)$ on sorted parking configurations.
Indeed, we decompose $\Psi$ as follows

\centerline{$\begin{array}{lcl}
 \Psi(f) & =  & T^{\mathrm{lastright}(w)-1-(n-3)} \circ (g\rightarrow g-\Delta^{(n)})^{n-3}\\
\ & \ &  \circ (g\rightarrow g+\epsilon^{(n)})^{n(n-3)}\circ(g\rightarrow \mathrm{sort}(-g))(f)\\
\end{array}$}

where $w$ is defined in the  doubly pointed cut skew cylinder $(C_n[w],s,0)$ related by $\mathrm{cyltoconf}$ to the sorted recurrent configuration $f$. 
Each operator in this decomposition is described in preceeding subsection via $\mathrm{cyltoconf}$, hence applying these descriptions  we also obtain 

\begin{statement} For any $w\in D_n$ and $s\in\mathbb{Z}$, we have 
$$\Psi(\mathrm{cyltoconf}(C_n[w],s,0)) = \mathrm{cyltoconf}(C_n[\Phi(w)],\mathrm{lastright}(w)-1-s,0) $$
\end{statement}

\begin{figure}[ht!]
\begin{tikzpicture}[scale=0.8]
\draw (-5,0) grid (9,7);
\draw[line width=2,draw=blue] (0,0) -- (0,2) -- (1,2) -- (1,5) -- (3,5) -- (3,6) -- (5,6) -- (5,7) -- (8,7);
\draw[line width=1,dashed,draw=green] (0,0) -- (7,7);
\draw[line width=0,fill=green] (0,0) circle (0.1);
\draw[line width=1,dashed,draw=red] (1,5) -- (-4,0);
\draw[line width=1,dashed,draw=red] (4,7) -- (2,5);
\draw[line width=0,fill=red] (1,5) circle (0.1);
\foreach \i/\j/\wlabel in {0/0/0,1/1/1,2/2/2,3/3/3,4/4/4,5/5/5,6/6/6,-1/0/7,0/1/8,1/2/9,2/3/10,3/4/11,4/5/12,5/6/13,-2/0/14,-1/1/15,0/2/16,1/3/17,2/4/18,3/5/19,4/6/20,-3/0/21,-2/1/22,-1/2/23,0/3/24,1/4/25,2/5/26,3/6/27,-4/0/28,-3/1/29,-2/2/30,-1/3/31,0/4/32,1/5/33}{
\draw node[green] at (\i-0.25,\j+0.25) {\tiny{$\wlabel$}};
}
\foreach \i/\j/\phiwlabel in {1/5/0,0/4/1,-1/3/2,-2/2/3,-3/1/4,4/7/5,3/6/6,2/5/7,1/4/8,0/3/9,-1/2/10,-2/1/11,5/7/12,4/6/13,3/5/14,2/4/15,1/3/16,0/2/17,-1/1/18,6/7/19,5/6/20,4/5/21,3/4/22,2/3/23,1/2/24,0/1/25,7/7/26,6/6/27,5/5/28,4/4/29,3/3/30,2/2/31,1/1/32,8/7/33}{
\draw node[red] at (\i+0.25,\j-0.25) {\tiny{$\phiwlabel$}};
}
\foreach \diag/\length in {-2/3,-1/6,0/6,1/6,2/6,3/5,4/4,5/3,6/2,7/1,8/0}{
  \foreach \j in {0,...,\length}{
   \draw[green,line width=2,->,opacity=0.4] (\diag+\j+0.1,\j+0.1) -- (\diag+\j+1-0.1,\j+1-0.1);
 }
}
\draw[fill=black,line width=0] (2,4) circle (0.12); 
\foreach \diag/\colength in {-2/4,-3/0,-4/0,-5/0,-6/1,-7/2,-8/3,-9/4,-10/5,-11/6}{
  \foreach \j in {\colength,...,6}{
   \draw[red,line width=2,->,opacity=0.4] (\diag+\j+0.9,\j+0.9) -- (\diag+\j+0.1,\j+0.1);
 }
}
\foreach \i/\j in {-1/0,-2/0,-1/1,0/2}{
\draw[line width=0,fill=red,opacity=0.4] (\i+0.5,\j+0.5) circle (0.15); 
}
\foreach \i/\j in {1/4,2/4,3/5}{
\draw[line width=0,fill=green,opacity=0.4] (\i+0.5,\j+0.5) circle (0.15); 
}
\end{tikzpicture}
\caption{A global view of the involution $\Psi$\label{fig:rrinvolution}}
\end{figure}
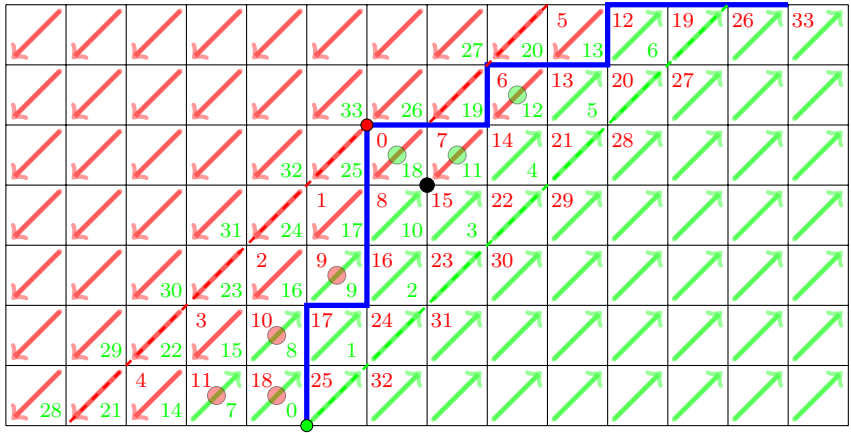

\subsubsection*{$\Psi$ and a refined superimposition.} 
This involution $\Psi$ can also be described globally by a refinement of the superimposition and in particular an approriate description of $rw(f)$ and $lw(f)$ from its doubly pointed cut skew cylinder $(C_n[w],s,0)$.

We associate to any doubly pointed cut skew cylinder $(C_n[w],s,0)$ the \emph{partial spiral traversal} which visits the vertices labeled from $-\infty$ to $s+n$.
In Figure~\ref{fig:rrinvolution}, $n=8$ and for $s=10$, the green arrows suggests the partial traversal from $-\infty$ to $s+n=18$, this last visited vertex labeled by $[18]_w$ being underlined by a black dot.
A cell is \emph{visited} by a partial traversal when at least three of its corners are visited otherwise it is called unvisited.
In Figure~\ref{fig:rrinvolution}, for $s+n=18$, the cells visited are those crossed by the green arrows.
Since the (blue) cut $w$ disconnect the skew cylinder into two connected components of cells adjacent by sides, we also consider the cells in the terminal (left) component of the complete spiral traversal and the complementary initial (right) component.
Hence a cell is either \emph{terminal} or \emph{initial} according to the component it belongs.
In this setting, the number $\mathrm{vister}(C_n[w],s,0)$ of visited and terminal cells corresponds to the cells used in Figure~\ref{fig:calcRank3} to compute the rank hence 
$$\mathrm{vister}(C_n[w],s,0) = \rho(f)+1.$$ 
On Figure~\ref{fig:rrinvolution}, these cells are marked by a red disk.
Similarly the number $\mathrm{unvini}(C_n[w],s,0)$ of unvisited and initial cells are related to the rank and degree of $f$ by 
$$\mathrm{unvini}(C_n[w],s,0)= {n-1\choose 2} + \rho(f)-deg(f).$$
On Figure~\ref{fig:rrinvolution}, these cells are marked by a green disk.

\begin{statement}
Let $f=\mathrm{cyltoconf}(C_n[w],s,0)$ be a sorted recurrent configuation, we have 
$$ (\mathrm{lw(f)},\mathrm{rw}(f)) = (\mathrm{vister}(C_n[w],s,0),\mathrm{unvini}(C_n[w],s,0)).$$
\end{statement}

The refinement of the superimposition of $C_n[w]$ and $\mathrm{shift}^{\mathrm{lastright}(w)+2n-1}\circ \mathrm{flip}(C_n[\Phi(w)])$ consider for the partial traversal of $C_n[w]$ ending in the vertex $[s+n]_w$, the almost complementary partial traversal of $C_n[\Phi(w)]$ also ending in the same vertex.
This vertex was underlined by the black dot in Figure~\ref{fig:rrinvolution} and the complementary partial traversal is described by the red arrows.
We notice that each cell is visited by exactly one of these two partial traversals.  
More precisely, we notice that this superimposition exchange visited and unvisited cells and also terminal and initial cells.
Hence the related involution exchanges $\mathrm{vister}$ and $\mathrm{univini}$ statistics.

The last vertex $[s+n]_w$ of the partial traversal related to  $(C_n[w],s,0)$ is labeled by $[\mathrm{lastright}(w)+2n-1-(s+n)]_{\Phi(w)}=[s'+n]_{\Phi(w)}$ in $C_n[\Phi(w)]$ where $s'$ is the undetermined value for the configuration defined in $C_n[\Phi(w)]$. 
Hence $s'=\mathrm{lastright}(w)-1-s$ and we notice that this involution exchanging $\mathrm{lw}(f)$ and $\mathrm{rw}(f)$ is indeed $\Psi$.  

On the example of Figure~\ref{fig:rrinvolution}, for $s=10$, the last visited vertex by the partial traversal in $C_n[w]$ is $s+n=[18]_w$. 
The label of this vertex in $C_n[\Phi(w)]$ is $[15]_{\Phi(w)} = [s'+8]_{\Phi(w)}$ so the superimposed configuration is defined on its sink by $s'=7$ which corresponds to $\mathrm{lastright}(w)-1-s=18-1-10=7$.

\end{document}